\documentclass{amsart} 
\usepackage[utf8]{inputenc}
\usepackage{amssymb,amsthm}
\usepackage{mathtools}\mathtoolsset{showonlyrefs}
\usepackage{microtype}
\usepackage[inline]{enumitem}
\usepackage{tabto} 
\usepackage{mathrsfs}
\usepackage{geometry}
\usepackage{tikz-cd}
\usepackage{csquotes} 
\usepackage{hyperref}

\newcommand{\Z}{\mathbb Z}
\newcommand{\Q}{\mathbb Q}
\newcommand{\R}{\mathbb R}
\newcommand{\F}{\mathbb F}
\newcommand{\field}{F}
\newcommand{\filtration}{\mathcal F}
\newcommand{\val}{\operatorname{val}}
\newcommand{\eg}{\textit{e.g.}, }
\newcommand{\ie}{\textit{i.e.}, }
\newcommand{\red}{\mathrm{red}}
\renewcommand{\setminus}{\smallsetminus}

\newcommand{\op}{\mathrm{op}}

\newcommand{\apartment}{\mathscr A}

\newcommand{\chamber}{\mathfrak C}
\newcommand{\hyperplanes}{\mathfrak H}

\newcommand{\X}{\mathrm X}
\newcommand{\aff}{\mathrm{aff}}
\newcommand{\vnu}{{}^\mathrm{v}\nu}
\newcommand{\pairing}{\langle\,\cdot\,{,}\,\cdot\,\rangle}
\newcommand{\scalar}{(\,\cdot\,{,}\,\cdot\,)}
\newcommand{\FGrp}{\mathbf{FGrp}}
\newcommand{\Grp}{\mathbf{Grp}}
\newcommand{\Hecke}{\mathcal H}
\newcommand{\Mod}[1]{\operatorname{Mod}\text{-}#1}
\newcommand{\module}[1]{\mathfrak{#1}}

\newcommand{\ol}[1]{\overline{#1}}
\newcommand{\alg}[1]{\mathbf{#1}}
\newcommand{\set}[2]{\left\{#1\,\middle|\,#2\right\}}
\newcommand{\wt}[1]{\widetilde{#1}}

\newcommand{\longtwoheadrightarrow}%
{\relbar\joinrel\twoheadrightarrow}
\newcommand{\longhookrightarrow}%
{\lhook\joinrel\relbar\joinrel\rightarrow}

\DeclareMathOperator{\Aut}{Aut}
\DeclareMathOperator{\diag}{diag}
\DeclareMathOperator{\End}{End}
\DeclareMathOperator{\GL}{GL}
\DeclareMathOperator{\gr}{gr}
\DeclareMathOperator{\Hom}{Hom}
\DeclareMathOperator{\id}{id}
\DeclareMathOperator{\Ind}{Ind}
\DeclareMathOperator{\Image}{Im}
\DeclareMathOperator{\Ker}{Ker}
\DeclareMathOperator{\pr}{pr}
\DeclareMathOperator{\Rep}{Rep}
\DeclareMathOperator*{\dotimes}{\otimes}

\newtheorem{prop}{Proposition}[section]

\newtheorem{thm}[prop]{Theorem}
\newtheorem*{thmintro}{Theorem}
\newtheorem{lem}[prop]{Lemma}
\newtheorem{cor}[prop]{Corollary}

\theoremstyle{definition}
\newtheorem*{defn}{Definition}
\newtheorem*{claim*}{Claim}

\theoremstyle{remark}
\newtheorem*{rmk}{Remark}
\newtheorem*{ex}{Example}
\newtheorem{rmk*}[prop]{Remark}
\newtheorem*{nota}{Notation}

\title{Parabolic Induction via the Parabolic pro-$p$ Iwahori--Hecke Algebra}
\author{Claudius Heyer}
\address{Mathematisches Institut, Westf\"alische Wilhelms-Universit\"at
M\"unster, Einsteinstra\ss{}e 62, D-48149 M\"unster, Germany}
\email{cheyer@uni-muenster.de}
\subjclass[2020]{11E95, 20C08, 20G25}

\begin{document}
\begin{abstract} 
Let $\mathbf{G}$ be a connected reductive group defined over a locally compact
non-archimedean field $F$, let $\mathbf{P}$ be a parabolic subgroup with Levi
$\mathbf{M}$ and compatible with a
pro-$p$ Iwahori subgroup of $G := \mathbf{G}(F)$. Let $R$ be a
commutative unital ring.

We introduce the parabolic pro-$p$ Iwahori--Hecke $R$-algebra
$\mathcal{H}_R(P)$ of $P := \mathbf{P}(F)$ and construct two $R$-algebra morphisms
$\Theta^P_M\colon \mathcal{H}_R(P)\to \mathcal{H}_R(M)$ and $\Xi^P_G\colon
\mathcal{H}_R(P) \to \mathcal{H}_R(G)$ into the pro-$p$ Iwahori--Hecke
$R$-algebra of $M := \mathbf{M}(F)$ and $G$, respectively. We prove that the
resulting functor $\Mod{\mathcal{H}_R(M)} \to \Mod{\mathcal{H}_R(G)}$ from the
category of right $\mathcal{H}_R(M)$-modules to the category of right
$\mathcal{H}_R(G)$-modules (obtained by pulling back via $\Theta^P_M$ and
extension of scalars along $\Xi^P_G$) coincides with the parabolic induction due
to Ollivier--Vign\'eras.

The maps $\Theta^P_M$ and $\Xi^P_G$ factor through a common subalgebra
$\mathcal{H}_R(M,G)$ of $\mathcal{H}_R(G)$ which is very similar
to $\mathcal{H}_R(M)$. Studying these algebras $\mathcal{H}_R(M,G)$ for varying
$(M,G)$ we prove a transitivity property for tensor products. As an application
we give a new proof of the transitivity of parabolic induction.
\end{abstract} 
\maketitle
\tableofcontents

\section{Introduction} 
\subsection{Motivation} 
Let $\alg G$ be a connected reductive group over a locally compact
non-archimedean field $\field$ with residue field of characteristic $p>0$. In
this introduction all representations will be on vector spaces over a fixed
field $k$. Given a parabolic subgroup $\alg P$ of $\alg G$
with Levi $\alg M$, parabolic induction is a process to obtain smooth
representations of $G\coloneqq \alg G(\field)$ from smooth representations of
$M\coloneqq \alg M(\field)$. More precisely, it is obtained as the composite
\begin{equation}\label{intro:parabolic-induction}
\begin{tikzcd}[row sep=0pt]
\Rep_k(M) \ar[r] & \Rep_k(P) \ar[r] & \Rep_k(G),\\
V \ar[r,mapsto] & V \ar[r,mapsto] & \Ind_P^G V.
\end{tikzcd}
\end{equation}
Here, $V$ is first viewed as a smooth representation of $P\coloneqq \alg
P(\field)$ by
letting the unipotent radical $U_P\coloneqq \alg U_{\alg P}(\field)$ of $P$ act
trivially.
Then we form the space $\Ind_P^GV$ of locally constant functions $f\colon G\to
V$ satisfying $f(\gamma g) = \gamma\cdot f(g)$ for all $\gamma\in P$ and $g\in
G$. Parabolic induction plays a fundamental role in classifying the smooth
admissible representations of $G$. 

A second important method to study smooth representations of $G$ is via the
functor of $K$-invariants $V\mapsto V^K$, for $K$ any compact open subgroup of
$G$. The ring of $G$-equivariant endomorphisms
\[
H_k(K,G) \coloneqq \End_{k[G]}\bigl(k[K\backslash G]\bigr)
\]
of the space $k[K\backslash G]$ of maps $K\backslash G\to k$ with finite
support acts on $V^K\cong \Hom_{k[G]}(k[K\backslash G], V)$ by
Frobenius reciprocity. In this way one
obtains a functor
\begin{equation}\label{eq:invariants}
\Rep_k(G) \longrightarrow \Mod{H_k(K,G)},\quad V\longmapsto V^K
\end{equation}
from $\Rep_k(G)$ into the category of right modules over the $k$-algebra
$H_k(K,G)$. If $k$ has characteristic $p$, it suffices to restrict attention to
a pro-$p$ Iwahori subgroup
$I_1$ of $G$; this is because for $K= I_1$ the functor \eqref{eq:invariants}
sends non-zero representations to non-zero modules. One is therefore led to
study the modules over the pro-$p$ Iwahori--Hecke algebra
$\Hecke_k(G)\coloneqq H_k(I_1,G)$.
A systematic study of $\Hecke_k(G)$ was carried out by Vign\'eras in
\cite{Vigneras.2005} (for $\field$-split $\alg G$) and \cite{Vigneras.2016} (for
arbitrary $\alg G$). 

Again, classifying the simple modules of $\Hecke_k(G)$ is accomplished by
studying induction functors
\[
\Mod{\Hecke_k(M)} \longrightarrow \Mod{\Hecke_k(G)},
\]
for parabolic subgroups $\alg P = \alg M\alg U_{\alg P}$ of $\alg G$ that are
compatible with $I_1$, which is
compatible with parabolic induction via \eqref{eq:invariants}. Such a functor
was defined and studied by Ollivier \cite{Ollivier.2010} for $\alg G =
\GL_n$, and for general $\alg G$ by Ollivier--Vign\'eras
\cite{Ollivier-Vigneras.2018}, Vign\'eras \cite{Vigneras.2015}, and
Abe~\cite{Abe.2016a,Abe.2016b}. Explicitly, this functor is given by
\begin{equation}\label{eq:hecke-induction}
\module m\longmapsto \module m\otimes_{\Hecke_k(M^+)} \Hecke_k(G),
\end{equation}
where $\Hecke_k(M^+)$ is a certain subalgebra of $\Hecke_k(M)$ embedding
naturally into $\Hecke_k(G)$. This algebra $\Hecke_k(M^+)$ depends on $I_1\cap
\alg U_{\alg P}(\field)$, hence also on $P$. However, the analogy between
\eqref{intro:parabolic-induction} and \eqref{eq:hecke-induction} is not as
strong as one might hope for. Let us recall the reason why the parabolic $P$
shows up in \eqref{intro:parabolic-induction}. One could directly define an
induction functor $\Rep_k(M)\to \Rep_k(G)$ by $V\mapsto \Ind_M^GV$. But this
functor is very difficult to handle, because the coset space $G/M$ and hence the
representation $\Ind_M^GV$ is ``too large''. On the other hand, the quotient
$G/P$ is compact which makes it possible to effectively study $\Ind_P^GV$. In
this light it is surprising that \eqref{eq:hecke-induction} is defined using the
small algebra $\Hecke_k(M^+)$. Instead, one would expect to use as big an algebra
as possible to define \eqref{eq:hecke-induction}; the parabolic Hecke algebra
$\Hecke_k(P) \coloneqq H_k(I_1\cap P,P)$ is a natural candidate although it
is not at all obvious how this can be achieved.

To summarize, this article is motivated by the following questions:
\begin{enumerate}[label=(Q\arabic*)]
\item\label{intro:i} Can we replace $\Hecke_k(M^+)$
in~\eqref{eq:hecke-induction} by the algebra $\Hecke_k(P)$?
\item\label{intro:ii} Assume the answer to~\ref{intro:i} is affirmative. Is the
functor
\[
\Mod{\Hecke_k(M)} \longrightarrow \Mod{\Hecke_k(G)},\quad \module m \longmapsto
\module m\otimes_{\Hecke_k(P)} \Hecke_k(G)
\]
naturally isomorphic to \eqref{eq:hecke-induction}?
\end{enumerate}
\subsection{Main results} 
We answer question~\ref{intro:i} positively even for $k$ an arbitrary
commutative unital ring. This amounts to constructing two $k$-algebra
homomorphisms $\Theta^P_M\colon \Hecke_k(P)\to \Hecke_k(M)$ and $\Xi^P_G\colon
\Hecke_k(P) \to \Hecke_k(G)$. 

The map $\Theta^P_M$ (see Proposition~\ref{prop:Theta}) exists quite generally
and is induced by the canonical projection map $k[I_1\cap P\backslash P]\to
k[I_1\cap M\backslash M]$. Its origins
can be traced back to the works of Andrianov, see, \eg
\cite[(3.3)]{Andrianov.1977} or \cite[Definition of $\Phi$ before
Chapter~3, Proposition~3.28]{Andrianov.1995}, who seems to have been the first
to study ``parabolic Hecke algebras'', albeit in a different context.

The main contribution of this work lies in the construction of $\Xi^P_G$.
Observing that the image of $\Theta^P_M$ contains $\Hecke_k(M^+)$, the
idea is to try to extend the embedding $\xi^+\colon \Hecke_k(M^+)\to
\Hecke_k(G)$ 
to $\Image(\Theta^P_M)$ and then to define $\Xi^P_G$ as the composite
$\Hecke_k(P)\xrightarrow{\Theta^P_M} \Image(\Theta^P_M)\to \Hecke_k(G)$.
However, since the goal for $\Xi^P_G\colon \Hecke_k(P)\to
\Hecke_k(G)$ is to have as large image as possible, this approach will not
always yield optimal results. Since in this approach $\Xi^P_G$ factors through a
subalgebra of $\Hecke_k(M)$, the best we can expect is for the image of
$\Xi^P_G$ to be canonically isomorphic to $\Hecke_k(M)$ as a $k$-module. 
If $k$ happens to be $p$-torsionfree, this is indeed the case. We prove:

\begin{thmintro}[Proposition~\ref{prop:Xi}] 
Assume that $k$ is $p$-torsionfree. The embedding $\xi^+\colon
\Hecke_k(M^+)\to \Hecke_k(G)$ extends to an injective $k$-algebra morphism
$\Image(\Theta^P_M)\to \Hecke_k(G)$. Moreover, this extension is
unique and $\Image(\Theta^P_M)$ is the maximal subalgebra of $\Hecke_k(M)$ with
this property.
\end{thmintro} 

Further, Corollary~\ref{cor:Image-basis} shows that $\Image(\Theta^P_M)$
identifies canonically with $\Hecke_k(M)$ as a $k$-module. 
If, on the other hand, $k$ is not $p$-torsionfree, then $\Image(\Theta^P_M)$ is
much smaller than $\Hecke_k(M)$. (For example, if $k$ is a field of
characteristic $p$, then $\Image(\Theta^P_M) = \Hecke_k(M^+)$.) Hence, even if
the assertion of Proposition~\ref{prop:Xi} could be proved without requiring
that $k$ be $p$-torsionfree, defining $\Xi^P_G$ as the composite $\Hecke_k(P)\to
\Image (\Theta^P_M)\to \Hecke_k(G)$ would yield a morphism with small image.
Instead, we make the following crucial definition:
\begin{defn} 
Let $k$ be arbitrary. We define 
\[
\Xi^P_G = \id_k\otimes \Xi^P_{G,\Z} \colon \Hecke_k(P)\to \Hecke_k(G).
\]
Here, $\Xi^P_{G,\Z}$ is the composite $\Hecke_\Z(P) \xrightarrow{\Theta^P_M}
\Image(\Theta^P_M) \to \Hecke_\Z(G)$, where the second arrow is the map in Proposition~\ref{prop:Xi} (for $k = \Z$).
\end{defn} 

In this way the image of $\Xi^P_G$ will always be large, that is, it can be
canonically identified with $\Hecke_k(M)$ as a $k$-module. The assumption on $k$
in Proposition~\ref{prop:Xi} is therefore inconsequential for the rest of the
paper.

Question~\ref{intro:ii} is a consequence of the construction of $\Theta^P_M$ and
$\Xi^P_G$, that is, we prove:
\begin{thmintro}[Theorem~\ref{thm:parabolic-induction}] 
The functor $\Mod{\Hecke_k(M)}\to \Mod{\Hecke_k(G)}$, $\module m\mapsto \module
m\otimes_{\Hecke_k(P)} \Hecke_k(G)$ is naturally isomorphic
to~\eqref{eq:hecke-induction}.
\end{thmintro} 

As an application we will give a new proof (see
Corollary~\ref{cor:transitivity}) of the transitivity of parabolic
induction originally due to Vign\'eras~\cite[Proposition~4.3]{Vigneras.2015}. 
\subsection{Structure of the paper} 
Section~\ref{sec:notations} is devoted to setting up the notation and reviewing
some parts of Bruhat--Tits theory. This part draws heavily from the original
paper~\cite{Bruhat-Tits.1972} by Bruhat--Tits and from the comprehensive summary
in \cite{Vigneras.2016}.
\medskip

In section~\ref{sec:groups} we will study the group index $\mu_{U_P}(g)
\coloneqq [I_{U_P} : I_{U_P} \cap g^{-1}I_{U_P}g]$ for $g\in P$, where $I_{U_P}
= I_1\cap U_P$. This index appears in the  double coset formula 
(Proposition~\ref{prop:double-coset}) which is used to give an explicit
description of the map $\Theta^P_M$ (Proposition~\ref{prop:Theta}). The results
of subsection~\ref{subsec:technical} are used to prove the estimate
$\mu_{U_P}(g)\ge \mu_{U_P}(g_M)$, where $g_M$ is the image of the projection of
$g$ in $M$, see Proposition~\ref{prop:inequality}. It allows us to give a
concrete
description of a basis of $\Image(\Theta^P_M)$ which will be necessary for the
construction of $\Xi^P_G$. In subsection~\ref{subsec:properties} we explain how
$\mu_{U_P}$ naturally gives rise to a function defined on the Iwahori--Weyl
group $W_M$ of $M$, again denoted $\mu_{U_P}$. Proposition~\ref{prop:mu-prop}
shows that $\mu_{U_P}$ measures how far the length function on $W_M$ deviates
from the length function on the Iwahori--Weyl group $W$ of $G$, and that
$\mu_{U_P}$ is compatible with the Bruhat order on $W_M$. These properties will
become useful in the construction of $\Xi^P_G$ (Proposition~\ref{prop:Xi}) and
in the study of the algebras $\Hecke_k(M,G)$ (section~\ref{sec:transitivity}).
We also obtain new and short proofs of two lemmas due to Abe in
Corollaries~\ref{cor:Abe-1} and~\ref{cor:Abe-2}.
\medskip

Section~\ref{sec:parabolic} contains the main results.
Subsection~\ref{subsec:reminder} gives a short introduction to abstract Hecke
algebras as double coset algebras following \cite[Chapter~3]{Andrianov.1995}.
In subsection~\ref{subsec:parabolic} we describe the map $\Theta^P_M$
explicitly, see Proposition~\ref{prop:Theta}. Together
with the inequality $\mu_{U_P}(g)\ge \mu_{U_P}(g_M)$ this yields a minimal
generating system of $\Image(\Theta^P_M)$ which is even a basis provided the
coefficient ring $R$ is $p$-torsionfree, see Corollary~\ref{cor:Image-basis}.
The main results on the construction of $\Xi^P_G$ and on parabolic induction are
proved in Proposition~\ref{prop:Xi} and Theorem~\ref{thm:parabolic-induction},
respectively.
\medskip

Finally, section~\ref{sec:transitivity} is devoted to applying our previous
results to give a new proof of the transitivity of parabolic induction. The
algebras $\Hecke_k(M,G)$ are introduced and studied in
subsection~\ref{subsec:definitions}. In Proposition~\ref{prop:localization} we
show that these algebras, for varying $G$, are localizations of each other. The
main result is the general Theorem~\ref{thm:tensor} from which we deduce that
parabolic induction is transitive. To finish, we construct in
subsection~\ref{subsec:filtration} a natural filtration on $\Hecke_k(M,G)$
(Proposition~\ref{prop:H(M,G)-filtration}) which might be of
independent interest.
\subsection{Acknowledgments} 
This article constitutes a part of my doctoral
dissertation \cite{Heyer.2019} at the Humboldt University in Berlin. It is a pleasure to thank my advisor Elmar
Gro\ss{}e-Kl\"onne for suggesting the topic and supporting me throughout the
years. My thanks extends to Peter Schneider for his interest in my work. I thank
the organizers of the conference ``Buildings and Affine Grassmannians'', held in
August--September 2019 at the CIRM, for giving me the chance to present my work. I
am especially grateful to the referee for carefully reading this article, for
pointing out a gap, and for providing several helpful comments and suggestions.
During the write-up of this article I was funded by the University of
M\"unster and Germany’s Excellence
Strategy EXC 2044  390685587, Mathematics Münster: Dynamics–Geometry–Structure.

\section{Notations and preliminaries}\label{sec:notations} 
Throughout the article we fix a locally compact non-archimedean field
$\field$ with residue field $\F_q$ of characteristic $p$ and normalized
valuation $\val_{\field}\colon \field \to \Z\cup \{\infty\}$. 

Given a group $G$, a subgroup $H \subseteq G$, and elements $g,h\in G$, we
write
\[
h^g\coloneqq g^{-1}hg,\qquad H^g\coloneqq \set{h^g}{h\in H},\qquad
H_{(g)}\coloneqq H\cap H^g.
\]
We also write $[g,h] \coloneqq ghg^{-1}h^{-1}$ for the commutator of $g$ and
$h$.

The symbol $\bigsqcup$ denotes ``disjoint union''.

Given an algebraic group $\alg H$ over $\field$, we denote by the corresponding
lightface letter $H\coloneqq \alg H(\field )$ its group of $\field$-rational
points. The topology of $\field$ makes $H$ into a topological group. We denote 
\begin{enumerate}[label=$-$]
\item $\alg H^{\circ}$ the identity component of $\alg H$; 
\item $\X^*(H)$ (resp. $\X_*(H)$) the group of algebraic $\field$-characters
(resp. $\field$-cocharacters) of $\alg H$.
\end{enumerate}

Let $\alg G$ be a connected reductive group defined over $\field$. Fix a maximal
$\field$-split torus $\alg T$ and denote by $\alg Z\coloneqq \alg Z_{\alg
G}(\alg T)$ (resp. $\alg N\coloneqq \alg N_{\alg G}(\alg T)$) the
centralizer (resp. normalizer) of $\alg T$ in $\alg G$. The finite
Weyl group $W_0\coloneqq N/Z$ acts on the (relative) root system $\Phi\coloneqq
\Phi(\alg G, \alg T) \subseteq \X^*(T)$ associated with $\alg T$; it identifies
with the Weyl group of $\Phi$. 

\begin{nota} 
Given a subset $\Psi \subseteq \Phi$, we denote $\Psi_\red \coloneqq
\set{\alpha\in \Psi}{\frac12 \alpha \notin \Psi}$ the set of \emph{reduced
roots} in $\Psi$.
\end{nota} 

Let $\alg U_{\alpha}$ be the root group associated with $\alpha\in \Phi$.
Then $\alg U_{2\alpha} \subseteq \alg U_{\alpha}$ whenever $\alpha,2\alpha\in
\Phi$. Fix a minimal parabolic $\field$-subgroup $\alg B$ with Levi
decomposition $\alg B = \alg Z\alg U$. This corresponds to a choice $\Phi^+$ of
positive roots
in $\Phi$, and we have $\alg U = \prod_{\alpha\in\Phi_{\red}^+}\alg U_\alpha$.
In this article parabolic subgroups are always standard and defined over
$\field$. By a Levi subgroup we mean the unique
Levi $\field$-subgroup $\alg M$ of $\alg P$ containing $\alg Z$; this is
expressed by writing $\alg P = \alg M\alg U_{\alg P}$, where $\alg U_{\alg P}$
denotes the unipotent radical of $\alg P$. Conversely, a Levi subgroup $\alg M$
in $\alg G$ determines a unique parabolic group $\alg P_{\alg M}$ with Levi
$\alg M$ and unipotent radical $\alg U_{\alg P_{\alg M}} = \prod_{\alpha\in
(\Phi^+\setminus\Phi_M)_{\red}} \alg U_{\alpha}$.

\subsection{The standard apartment} 
Consider the finite-dimensional $\R$-vector space
\[
V \coloneqq \R\otimes \X_*(T)/\X_*(C),
\]
where $\alg C$ denotes the connected center of $\alg G$. The finite Weyl group
$W_0$ acts on $V$ via the conjugation action on $T$, and the natural pairing
$\pairing\colon V^*\times V \to \R$ is $W_0$-invariant, where $V^*$ denotes the
$\R$-linear dual of $V$. Fix a $W_0$-invariant
scalar product $\scalar$ on $V$. The root system $\Phi$ embeds into $V^*$. 
For each $\alpha\in \Phi$ there exists a unique coroot $\alpha^\vee\in V$ with
$\langle \alpha, \alpha^\vee\rangle = 2$ such that the reflection
\[
s_{\alpha}\colon V^*\longrightarrow V^*,\quad x\longmapsto x - \langle x,
\alpha^\vee\rangle\cdot \alpha
\]
leaves $\Phi$ invariant.

\subsubsection{Valuations} 
A valuation $\varphi = (\varphi_\alpha)_{\alpha\in \Phi}$ on the
root group datum $(Z, (U_\alpha)_{\alpha\in \Phi})$ consists of a family of
functions $\varphi_{\alpha}\colon U_\alpha\to \R \cup \{\infty\}$ satisfying a
list of axioms \cite[(6.2.1)]{Bruhat-Tits.1972}. 
Then $\varphi$ is called \emph{discrete} if the set of values
$\Gamma_\alpha\coloneqq \varphi_\alpha(U_\alpha\setminus\{1\}) \subseteq \R$ is
discrete. 
It is called \emph{special} if $0\in \Gamma_\alpha$ for all
$\alpha\in\Phi_{\red}$. We write
\[
\Gamma'_\alpha \coloneqq
\set{\varphi_{\alpha}(u)}{\text{$u\in U_\alpha\setminus\{1\}$ and
$\varphi_{\alpha}(u) = \sup \varphi_{\alpha}(uU_{2\alpha})$}}.
\]
Then $\Gamma'_\alpha = \Gamma_\alpha$ if $2\alpha\notin \Phi$ and $\Gamma_\alpha
= \Gamma'_\alpha \cup (\frac12 \Gamma_{2\alpha})$ for all $\alpha\in \Phi$
\cite[(6.2.2)]{Bruhat-Tits.1972}.

The group $N$ acts naturally on the set of valuations via
\[
(n.\varphi)_\alpha(u) = \varphi_{w^{-1}(\alpha)}(n^{-1}un),\qquad \text{for
$u\in U_\alpha$, $n\in N$,}
\]
where $w\coloneqq \vnu(n) \in W_0 = N/Z$ and $\vnu\colon N\to N/Z = W_0$ denotes
the projection map. 

The vector space $V$ acts faithfully on the set of valuations via
\[
(\varphi+v)_\alpha(u) = \varphi_\alpha(u) + \langle \alpha,v\rangle,\qquad
\text{for $u\in U_\alpha$, $v\in V$.}
\]
One easily verifies $n.(\varphi+v) = n.\varphi + \vnu(n)(v)$ for $n\in N$, $v\in
V$ \cite[(6.2.5)]{Bruhat-Tits.1972}. 

Given $z\in Z$, there is a unique vector $\nu(z)\in V$ satisfying $z.\varphi =
\varphi + \nu(z)$ \cite[Proposition~(6.2.10), Proof of (i)]{Bruhat-Tits.1972}.
In this way we obtain a group homomorphism
\begin{equation}\label{eq:nu}
\nu\colon Z \longrightarrow V,\quad z\longmapsto \nu(z).
\end{equation}

Restriction of characters realizes $\X^*(Z)$ as a subgroup of finite index in
$\X^*(T)$ \cite[V.2.6. Lemme]{Renard.2010}. Therefore, given $\chi\in \X^*(T)$,
there exists $n\in\Z_{>0}$ with $n\chi \in \X^*(Z)$, and one easily verifies
that the definition
\[
(\val_{\field}\circ \chi)(z) \coloneqq \frac1n\cdot
\val_{\field}\bigl((n\chi)(z)\bigr) \in \R,\qquad \text{for $z\in Z$},
\]
is independent of $n$. We say that $\varphi$ is \emph{compatible with
$\val_{\field}$} if 
\[
\bigl\langle \alpha, \nu(z)\big\rangle = -(\val_{\field}\circ \alpha)(z),\qquad
\text{for all $z\in Z$, $\alpha\in\Phi$.}
\]
\subsubsection{The apartment, hyperplanes, and affine roots} 
From now on we fix a discrete, special valuation $\varphi_0 =
(\varphi_{0,\alpha})_{\alpha\in \Phi}$ on $(Z, (U_\alpha)_{\alpha\in \Phi})$
which is compatible with $\val_{\field}$; it exists by \cite[5.1.20. Theor\`eme
and 5.1.23. Proposition]{Bruhat-Tits.1984}. By \cite[(37)]{Vigneras.2016} the
subgroups
\[
U_{\alpha,r} \coloneqq \set{u\in U_\alpha}{\varphi_{0,\alpha}(u) \ge r},\qquad
\text{for $r\in \R$},
\]
form a basis of compact open neighborhoods of the neutral element in $U_\alpha$,
$\alpha\in \Phi$. 

The \emph{apartment} of $G$ is defined as the affine space
\[
\apartment \coloneqq \set{\varphi_0 + v}{v\in V}
\]
under $V$. It follows from \cite[Proposition (6.2.10)]{Bruhat-Tits.1972} that
the action of $N$ on valuations induces a group homomorphism
\[
\nu\colon N \longrightarrow \Aut \apartment
\]
extending \eqref{eq:nu}.
The fixed $W_0$-invariant scalar
product on $V$ endows $\apartment$ with a Euclidean metric. 

Given $\alpha\in V^*$ and $r\in \R$, we put
\begin{align*}
a_{\alpha,r} &\coloneqq \set{\varphi_0+v\in \apartment}{\langle \alpha,
v\rangle + r \ge 0}\quad \text{and}\\
H_{\alpha,r} &\coloneqq \set{\varphi_0+v\in \apartment}{\langle \alpha, v\rangle
+ r = 0}.
\end{align*}
We call $\Phi^{\aff} \coloneqq \set{a_{\alpha,r}}{\text{$\alpha\in \Phi$ and
$r\in \Gamma'_\alpha$}}$ the set of \emph{affine roots} of $\apartment$ and
denote
\[
\hyperplanes \coloneqq \set{H_{\alpha,r}}{\text{$\alpha\in \Phi$ and $r\in
\Gamma'_\alpha$}} = \set{H_{\alpha,r}}{\text{$\alpha\in \Phi_{\red}$ and $r\in
\Gamma_\alpha$}}
\]
the set of hyperplanes in $\apartment$. 
A connected component of $\apartment \setminus \bigcup\hyperplanes$ is called a
\emph{chamber}. 
Associated with $\varphi_0$ and $\alg B$ there is a unique chamber $\chamber$
determined by $\varphi_0\in \ol{\chamber}$ (topological closure) and 
\[
\chamber \subseteq \set{\varphi_0+v\in \apartment}{\langle \alpha, v\rangle
>0,\quad \text{for all $\alpha\in\Phi^+$}}.
\]
We call $\chamber$ the \emph{fundamental chamber}. 

The action of $N$ on $\apartment$ induces natural actions on $\Phi^\aff$ and on
$\hyperplanes$. Explicitly, we have
\[
n.a_{\alpha,r} = a_{w(\alpha), r - \langle w(\alpha), n.\varphi_0 -
\varphi_0\rangle},
\]
for all $a_{\alpha,r}\in \Phi^\aff$ and $n\in N$, where $w\coloneqq \vnu(n)\in
W_0$. A similar formula holds for $H_{\alpha,r}$. Likewise,
\begin{equation}\label{eq:U-action}
nU_{\alpha,r}n^{-1} = U_{w(\alpha), r - \langle w(\alpha), n.\varphi_0 -
\varphi_0\rangle},
\end{equation}
for all $\alpha\in \Phi$, $r\in \R$, $n\in N$.
\subsubsection{The affine Weyl group} 
Denote $s_H \in \Aut\apartment$ the orthogonal reflection through $H\in
\mathfrak H$ and put 
\[
S(\hyperplanes) \coloneqq \set{s_H}{H\in \hyperplanes}.
\]
Conversely, we denote $H_s\in \hyperplanes$ the hyperplane fixed by $s\in
S(\hyperplanes)$. The \emph{affine Weyl group} $W^\aff$ is the subgroup of $\wt
W\coloneqq \nu(N) \subseteq \Aut \apartment$ generated by $S(\hyperplanes)$.
Notice that $ws_Hw^{-1} = s_{w(H)}$ and $wH_s = H_{wsw^{-1}}$ for all $w\in \wt
W$, $s\in S(\hyperplanes)$, and $H\in \hyperplanes$. 
We denote by $S^{\aff}$ the set of reflections in the walls of
the fundamental chamber $\chamber$. It generates $W^\aff$ as a group. Moreover,
$W^\aff$ acts simply transitively on the set of chambers of $\apartment$. 

The stabilizer
$W_{\varphi_0}$ of $\varphi_0$ in $\wt W$ identifies with $W_0$, because
$\varphi_0$ is special. We obtain semidirect product decompositions
\[
\wt W = W_0\ltimes (\wt W\cap V)\qquad \text{and}\qquad W^\aff = W_0 \ltimes
(W^\aff \cap V),
\]
and $W^\aff \cap V$ is generated by the translations $r\alpha^\vee$, for
$\alpha\in \Phi_\red$ and $r\in \Gamma_\alpha$ \cite[Proposition
(6.2.19)]{Bruhat-Tits.1972}.

By \cite[Proposition~(6.2.22)]{Bruhat-Tits.1972} there exists a unique reduced
root system $\Sigma$ in $V^*$ such that $W^\aff$ is the affine Weyl group of
$\Sigma$, \ie it is the subgroup of $\Aut \apartment$ generated by the
reflections
\[
s_{\alpha,k} \colon \Aut \apartment \longrightarrow \Aut \apartment,\quad
\varphi_0 + v \longmapsto \varphi_0 + v - \bigl(\langle \alpha, v\rangle +
k\bigr)\cdot \alpha^\vee,
\]
for $(\alpha,k) \in \Sigma^\aff \coloneqq \Sigma\times\Z$.
By a suitable scaling we obtain a surjective map
$\Phi \to \Sigma$, $\alpha\mapsto \varepsilon_\alpha \alpha$
between root systems which induces a bijection $\Phi_{\red}\cong \Sigma$. By
\cite[(39) ff]{Vigneras.2016} we have $\varepsilon_\alpha =
\varepsilon_{-\alpha} \in \Z_{>0}$ and $\Gamma_{\alpha} =
\varepsilon_\alpha^{-1} \Z$, for $\alpha\in \Phi_{\red}$, is a group.
\begin{nota}
In order to avoid confusion when working with the two root systems $\Phi_{\red}$
and $\Sigma$ we will write $H_{(\alpha,k)}$ and $U_{(\alpha,k)}$ instead of
$H_{\alpha,k}$ and $U_{\beta, \varepsilon_\beta^{-1}k}$ whenever $\beta\in
\Phi_{\red}$, $\alpha = \varepsilon_\beta\beta$, and $k\in\Z$. 
\end{nota}

Given $n\in N$
with image $w$ in $\wt W$ and $(\alpha,k)\in \Sigma^\aff$, we have
$wH_{(\alpha,k)} = H_{w\cdot (\alpha,k)}$ and 
\begin{equation}\label{eq:action-U}
nU_{(\alpha,k)}n^{-1} = U_{w\cdot (\alpha,k)}.
\end{equation}

\subsection{Parahoric subgroups} 
The pointwise stabilizer of $\varphi_0$, resp. $\chamber$, in the kernel of the
Kottwitz homomorphism $\kappa_G$ \cite[7.1 to 7.4]{Kottwitz.1997} is denoted by
$K$, resp. $I$. We call $I$ the \emph{Iwahori subgroup}; its pro-$p$ Sylow
subgroup $I_1$ is called the \emph{pro-$p$ Iwahori subgroup}. Both $K$ and $I$
are examples of \emph{parahoric subgroups} \cite{Haines-Rapoport.2008}; they are
compact open subgroups of $G$.

We remark that $K$ is a special parahoric subgroup containing $I$, and it
satisfies \cite[(51)]{Vigneras.2016}
\[
K\cap U_{\alpha} = U_{(\alpha,0)},\qquad \text{for all $\alpha\in \Phi$.}
\]

Put $Z_0 \coloneqq Z\cap K = Z\cap I$ \cite[Lemma~4.2.1]{Haines-Rostami.2009}
with pro-$p$ radical $Z_1= Z\cap I_1$. Then $Z_0$ is the unique parahoric
subgroup of $Z$, and $N$ normalizes both $Z_0$ and $Z_1$. The multiplication map
\begin{equation}\label{eq:I1-homeo}
\prod_{\alpha\in -\Sigma^+} U_{(\alpha,1)} \times Z_1 \times \prod_{\alpha\in
\Sigma^+} U_{(\alpha,0)} \xrightarrow{\cong} I_1
\end{equation}
is a homeomorphism \cite[Corollary~3.20]{Vigneras.2016} with respect to any
ordering of the factors.

\subsection{The Iwahori--Weyl group} 
We call
\[
W\coloneqq N/Z_0,\qquad \text{resp.}\qquad W(1)\coloneqq N/Z_1
\]
the \emph{Iwahori--Weyl group}, resp. \emph{pro-$p$ Iwahori--Weyl group}. There
are exact sequences
\[
1 \longrightarrow Z_0/Z_1 \longrightarrow W(1) \longrightarrow W \longrightarrow
1
\]
and
\begin{equation}\label{eq:Lambda-Weyl}
0 \longrightarrow \Lambda \longrightarrow W \longrightarrow W_0 \longrightarrow
1,
\end{equation}
where $\Lambda\coloneqq Z/Z_0$ is a finitely generated abelian group with finite
torsion and the same rank as $\X_*(T)$
\cite[Theorem~1.0.1]{Haines-Rostami.2009}. It is thus written
additively. When viewed as an element of $W$ we use the exponential notation
$e^\lambda$ rather than $\lambda\in \Lambda$ in order to avoid confusion.

Given a subset $X\subseteq W$, we denote $X(1)$ the preimage of $X$ under the
projection $W(1)\to W$.

We remark that the sequence~\eqref{eq:Lambda-Weyl} splits, providing a
semidirect product decomposition
\[
W = \Lambda \rtimes W_0.
\]
In particular, $W_0$ acts on $\Lambda$ via $w(\lambda) = we^{\lambda}w^{-1}$.
The group $\Lambda(1)$ is not abelian in general. Notice that \eqref{eq:nu}
factors through $\Lambda$ (and hence $\Lambda(1)$).

The inclusion $N\hookrightarrow G$ induces bijections
\begin{align}\label{eq:Bruhat-Decomposition}
W_0 &\cong B\backslash G/B, & W&\cong I\backslash G/I, & W(1) &\cong
I_1\backslash G/I_1.
\end{align}

The group $\Omega = \set{u\in W}{u\chamber = \chamber}$ is abelian and acts on
$S^\aff$ by conjugation, and we have a decomposition
\[
W = W^\aff \rtimes \Omega.
\]
The length function $\ell$ on the Coxeter group $(W^\aff,S^\aff)$ extends to a
length function $\ell$ on
$W$ if we define $\ell(wu) = \ell(w)$ for $w\in W^\aff$, $u\in \Omega$. By
inflation we also obtain a length function $\ell$ on $W(1)$.

We denote by $\le$ the Bruhat order on $W^\aff$. It extends to the Bruhat order
$\le$ on $W$ if we put
\[
wu \le w'u'\iff w\le w'\text{ and }u = u',\qquad \text{for $w,w'\in
W^\aff$, $u,u'\in \Omega$.}
\]
We define $v<w$ as $v\le w$ and $v\neq w$. By inflation, we obtain the Bruhat
order $<$ on $W(1)$: given $\wt v, \wt w\in W(1)$ with
image $v,w\in W$, respectively, we define
\[
\wt v < \wt w \iff v < w.
\]

\subsection{The integers \texorpdfstring{$q_w$}{qw}} 
Given $n\in N$ with image $w$ in $W$ (or $W(1)$), one defines
\[
q_w\coloneqq \big\lvert InI/I\big\rvert = \big\lvert I_1nI_1/I_1\big\rvert.
\]
An application of \cite[Proposition~3.38]{Vigneras.2016} shows
\[
q_w = q_{s_1}\dotsm q_{s_{\ell(w)}}
\]
whenever $w = s_1\dotsm s_{\ell(w)}u$ with $s_i\in
S^\aff$ and $u\in \Omega$ (or $s_i\in S^\aff(1)$ and $u\in \Omega(1)$). Given
$s\in S^\aff$, write $H_s = H_{\beta,r} = H_{(\alpha,k)}$ with $\beta\in
\Phi^+_\red$, $r\in \Gamma_\beta$ and $\alpha = \varepsilon_\beta\beta\in
\Sigma^+$, $k = \varepsilon_\beta r\in\Z$. Then
\begin{equation}\label{eq:qs}
q_s = \big\lvert U_{\beta,r}/U_{\beta,r+}\big\rvert = \big\lvert
U_{(\alpha,k)}/U_{(\alpha,k+1)}\big\rvert,
\end{equation}
where $U_{\beta,r+} \coloneqq \bigcap_{r'>r} U_{\beta,r'}$. Indeed, if $r\in
\Gamma_\beta \setminus \Gamma'_\beta$, then $2\beta\in\Phi^+$ and $2r\in
\Gamma'_{2\beta}$. In this case we have $U_{\beta,r+}\cdot U_{2\beta,2r} =
U_{\beta,r}$, and \cite[Lemma~3.8]{Vigneras.2016} yields an isomorphism
\[
U_{\beta,r}/U_{\beta,r+} \cong U_{2\beta,2r}/U_{2\beta,2r+}.
\]
Now, the first equality in \eqref{eq:qs} follows from
\cite[Corollary~3.31]{Vigneras.2016} while the second equality is clear.

Notice that $q_s = q_{s'}$ whenever $s,s'\in S^\aff$ are conjugate in $W$
\cite[(67)]{Vigneras.2016}. As every hyperplane in $\hyperplanes$ is of the form
$w H_s$ for some $w\in W$, $s\in S^\aff$, we obtain a well-defined function
\begin{equation}\label{eq:q-hyperplanes}
\hyperplanes\longrightarrow q^{\Z_{>0}},\quad q(wH_s) \coloneqq q_s.
\end{equation}
For every $w\in W$ (or $W(1)$) we then have
\cite[Definition~4.14]{Vigneras.2016}
\begin{equation}\label{eq:qw-hyperplanes}
q_w = \prod_{H\in \hyperplanes_w}q(H),
\end{equation}
where $\hyperplanes_w$ denotes the set of hyperplanes in $\hyperplanes$
separating $\chamber$ and $w\chamber$.

For $v,w\in W$ (or $W(1)$) there exists a unique integer $q_{v,w}\in
q^{\Z_{\ge0}}$ satisfying \cite[Definition~4.14]{Vigneras.2016}
\[
q_vq_w = q_{vw}q_{v,w}^2.
\]
Notice that
\begin{equation}\label{eq:qvw-hyperplanes}
q_{v,w} = \prod_{H\in \hyperplanes_v\cap v\hyperplanes_w} q(H).
\end{equation}
(In \cite[Lemma~4.19]{Vigneras.2016} this is only stated for $v,w\in W^\aff$.
But this follows in general from the facts that $\hyperplanes_{wu} =
\hyperplanes_w$ and $u\hyperplanes_w = \hyperplanes_{uwu^{-1}}$ whenever $w\in
W^\aff$ and $u\in \Omega$.)

\begin{rmk} 
\begin{enumerate}[label=(\alph*)]
\item If $\alg G$ is $\field$-split, then $q_w = q^{\ell(w)}$ for all $w\in W$.
For this reason the function $w\mapsto q_w$ may be viewed as a generalized
length function on $W$.
\item In general, $q_{v,w} = 1$ if and only if $\ell(vw) = \ell(v) + \ell(w)$.
\end{enumerate}
\end{rmk} 

\subsection{Levi subgroups} 
Let $\alg P = \alg M \alg U_{\alg P}$ be a (standard) parabolic subgroup. Then
$\alg Z \subseteq \alg M$ and $\alg N_{\alg M}\coloneqq \alg N_{\alg M}(\alg T)
= \alg N\cap\alg M$. 
All the objects we have defined for $\alg G$ have an analogue for
$\alg M$ and we denote them by attaching the index $M$; for example, we write
$\Phi_M$, $W_{0,M}$, $\apartment_M$, $\hyperplanes_M$, $q_{M,w}$, etc.

Notice that $W_{0,M}$ is contained in $W_0$, $W_M$ identifies with
$\Lambda\rtimes W_{0,M}$, and $W_M(1)$ is the preimage of $W_M$ under $W(1)\to
W$. The restriction $\varphi_{0,M}$ of $\varphi_0$ to 
$(Z, (U_{\alpha})_{\alpha\in\Phi_M})$ is again discrete, special, and
compatible with $\val_{\field}$. The $\R$-vector space
\[
V_M \coloneqq \R\otimes X_*(T)/X_*(C_M),
\]
where $\alg C_{\alg M} = \bigl(\bigcap_{\alpha\in
\Phi^+_M}\Ker\alpha\bigr)^\circ$ is the connected center of $\alg M$, is a
quotient of $V$. Then $\Phi_M = V_M^*\cap \Phi$ and $\Sigma_M = V_M^*\cap
\Sigma$. The projection $V\twoheadrightarrow V_M$ induces a natural
$N_M$-equivariant projection
\[
p_M\colon \apartment \longtwoheadrightarrow \apartment_M
\]
sending $\varphi_0 \mapsto \varphi_{0,M}$. Taking inverse images we obtain
$N_M$-equivariant inclusions $\Phi_M^\aff \subseteq \Phi^\aff$ and
$\hyperplanes_M \subseteq \hyperplanes$. Therefore, we have also
$S(\hyperplanes_M)\subseteq S(\hyperplanes)$ and hence $W_M^\aff \subseteq
W^\aff$. 
\begin{rmk} 
In general we have only $p_M(\chamber) \subseteq \chamber_M$ and not an
equality. In this case we have $S_M^\aff \nsubseteq S^\aff$. Therefore, the
length and Bruhat order on $W_M^\aff$ are not obtained by restricting the length
and Bruhat order of $W^\aff$.
\end{rmk} 

\section{Groups, double cosets, and indices}\label{sec:groups} 
We fix a (standard) parabolic $\field$-subgroup $\alg P = \alg M\alg U_{\alg P}$
of $\alg G$.
\subsection{A double coset formula}\label{subsec:double} 
Let $\Gamma \subseteq P$ be a compact open subgroup satisfying $\Gamma =
\Gamma_M\Gamma_{U_P}$, where $\Gamma_M \coloneqq \Gamma\cap M$ and
$\Gamma_{U_P} \coloneqq \Gamma \cap U_P$. This means that every $g\in \Gamma$
can be uniquely written as
\[
g = g_{M}\cdot g_{U_P},\qquad \text{for some $g_{M}\in \Gamma_{M}$, $g_{U_P}\in
\Gamma_{U_P}$.}
\]
Notice that for all $g\in P$ and $m\in M$ the indices
\[
\mu(g)\coloneqq [\Gamma : \Gamma_{(g)}],\qquad \mu_{U_P}(g)\coloneqq
[\Gamma_{U_P} : (\Gamma_{U_P})_{(g)}],\qquad \mu_M(m)\coloneqq [\Gamma_M :
(\Gamma_M)_{(m)}]
\]
are finite, because $\Gamma\subseteq P$ is compact open. We also remark that the
projection map $\pr_M\colon P\twoheadrightarrow M$ is continuous and open, so
that $\nu_M(g)\coloneqq [(\Gamma_M)_{(g_M)} : \pr_M(\Gamma_{(g)})]$ is finite.

The following proposition generalizes \cite[Lemma~2]{Gritsenko.1988}.

\begin{prop}\label{prop:double-coset} 
Given $g\in P$, consider the coset decompositions
\[
\Gamma_M = \bigsqcup_{i=1}^{\mu_M(g_M)} (\Gamma_M)_{(g_M)} m_i,\quad
(\Gamma_M)_{(g_M)} = \bigsqcup_{j=1}^{\nu_M(g)} \pr_M(\Gamma_{(g)}) h_j,\quad
\Gamma_{U_P} = \bigsqcup_{s=1}^{\mu_{U_P}(g)} (\Gamma_{U_P})_{(g)}u_s.
\]
Then one has a decomposition of the double coset
\begin{equation}\label{eq:double-coset}
\Gamma g\Gamma = \bigcup_{i=1}^{\mu_M(g_M)} \bigcup_{j=1}^{\nu_M(g)}
\bigcup_{s=1}^{\mu_{U_P}(g)} \Gamma gu_sh_jm_i.
\end{equation}
Moreover, the union is disjoint.
In particular, $\mu(g) = \mu_M(g_M)\cdot \nu_M(g)\cdot \mu_{U_P}(g)$.
\end{prop} 
\begin{proof} 
It is clear that the right hand side of \eqref{eq:double-coset} is contained in
$\Gamma g\Gamma$. For the converse inclusion let $\gamma = \gamma_M\gamma_{U_P}
\in \Gamma$. Write 
\begin{align*}
\gamma_M &= mm_i, & & \text{for some $m\in (\Gamma_M)_{(g_M)}$ and
$1\le i\le \mu_M(g_M)$,}\\
m &= hh_j, & & \text{for some $h\in \pr_M(\Gamma_{(g)})$ and $1\le j\le
\nu_M(g)$.}
\end{align*}
By definition of $h$ there exists $u\in \Gamma_{U_P}$ with $hu^{-1} \in
\Gamma_{(g)}$, \ie with $ghu^{-1} = yg$ for some $y\in \Gamma$. Thus,
\[
\gamma = \gamma_M \gamma_{U_P} = hh_jm_i\gamma_{U_P} = hu^{-1}\cdot
u\gamma_{U_P}^{(h_jm_i)^{-1}}\cdot h_jm_i.
\]
As $\Gamma_M$ normalizes $\Gamma_{U_P}$ we may write
$vu_s = u\gamma_{U_P}^{(h_jm_i)^{-1}} \in \Gamma_{U_P}$ for some $v\in
(\Gamma_{U_P})_{(g)}$ and some integer $1\le s\le \mu_{U_P}(g)$. Then also
$v^{g^{-1}} \in \Gamma_{U_P}$, and therefore
\begin{align*}
g\gamma &= ghu^{-1}\cdot u\gamma_{U_P}^{(h_jm_i)^{-1}}\cdot h_jm_i\\
&= yg\cdot vu_s\cdot h_jm_i\\
&= yv^{g^{-1}}\cdot g\cdot u_sh_jm_i \in \Gamma g u_sh_jm_i.
\end{align*}
This proves equality in \eqref{eq:double-coset}. To see disjointness in
\eqref{eq:double-coset}, assume $gu_sh_jm_i = \gamma gu_th_am_b$ for some
$\gamma\in \Gamma$. Rearranging gives
\begin{equation}\label{eq:double-coset-2}
\gamma^g = u_sh_jm_im_b^{-1}h_a^{-1}u_t^{-1} \in \Gamma_{(g)}.
\end{equation}
Applying $\pr_M$ to \eqref{eq:double-coset-2} yields
\begin{equation}\label{eq:double-coset-3}
\gamma_M^{g_M} = h_jm_im_b^{-1} h_a^{-1} \in \pr_M\bigl(\Gamma_{(g)}\bigr).
\end{equation}
In particular, $m_im_b^{-1} = h_j^{-1} \gamma_M^{g_M} h_a \in
(\Gamma_M)_{(g_M)}$, whence $i=b$. Therefore, equation \eqref{eq:double-coset-3}
reads $\gamma_M^{g_M} = h_jh_a^{-1} \in \pr_M\bigl(\Gamma_{(g)}\bigr)$, and we
deduce $j=a$. Going back to \eqref{eq:double-coset-2} gives $\gamma^g = u_s
u_t^{-1} \in \Gamma_{(g)} \cap U_P$. But notice that $\Gamma_{(g)}\cap U_P =
(\Gamma_{U_P})_{(g)}$, because $g$ normalizes $U_P$. Consequently,
$u_su_t^{-1}\in (\Gamma_{U_P})_{(g)}$, whence $s=t$. This concludes the proof of
the disjointness assertion. 

The map $\Gamma_{(g)}\backslash \Gamma \to \Gamma\backslash \Gamma g\Gamma$,
sending $\Gamma_{(g)}\gamma \mapsto \Gamma g\gamma$, is well-defined and
bijective. Hence, we have $\mu(g) = [\Gamma : \Gamma_{(g)}] = \lvert
\Gamma\backslash \Gamma g\Gamma\rvert$ and the last assertion follows.
\end{proof} 

\begin{rmk} 
In general, we have $\nu_M(g)\neq 1$, \ie $\pr_M(\Gamma_{(g)}) \subsetneqq
(\Gamma_M)_{(g_M)}$. As a concrete example consider the group $P$ of upper
triangular matrices inside $\GL_2(\Q_p)$. It contains the subgroup $M$ of
diagonal matrices and the subgroup $U_P$ of upper triangular unipotent matrices.
Let $\Gamma = \begin{psmallmatrix} (1+p\Z_p)^\times & \Z_p\\ 0 &
(1+p\Z_p)^\times\end{psmallmatrix}$ and $g\coloneqq \begin{psmallmatrix}1 &
p^{-n-1}\\0&1\end{psmallmatrix}$ for some integer $n\in\Z_{\ge0}$. Then $g_M =
1$, whence $(\Gamma_M)_{(g_M)} = \Gamma_M$. Given any $\gamma =
\begin{psmallmatrix} 1+pa & b\\0&1+pc\end{psmallmatrix}$ in $\Gamma$, we compute
$g\gamma g^{-1} = \begin{psmallmatrix} 1+pa & b+p^{-n}(c-a)\\0 &
1+pc\end{psmallmatrix}$.
Therefore, $g\gamma g^{-1}\in \Gamma$ if and only if $c-a\in p^n\Z_p$. Thus,
$\Gamma_{(g)} = \set{\begin{psmallmatrix}1+pa & b\\0 &
1+pa+p^{n+1}c\end{psmallmatrix}} {a,b,c\in \Z_p}$,
so that 
\[
\pr_M\bigl(\Gamma_{(g)}\bigr) = \set{\begin{pmatrix} 1+pa & 0\\0 &
1+pa+p^{n+1}c\end{pmatrix}}{a,c\in\Z_p}.
\]
From this description it is already clear that $\nu_M(g)\neq 1$ in general. As
an exercise, and in order to illustrate the methods employed in
section~\ref{subsec:inequality}, we explicitly compute $\nu_M(g)$.
Consider the reduction modulo $p^{n+1}$ map
\[
\psi\colon \Gamma_M = \begin{pmatrix} (1+p\Z_p)^\times & 0\\0 &
(1+p\Z_p)^\times\end{pmatrix} \longrightarrow \begin{pmatrix}
(\Z_p/p^{n+1}\Z_p)^\times & 0\\0 & (\Z_p/p^{n+1}\Z_p)^\times\end{pmatrix}.
\]
Its kernel is contained in $\pr_M\bigl(\Gamma_{(g)}\bigr)$, and we have
\begin{align*}
\psi\bigl(\Gamma_M\bigr) &= \set{\begin{pmatrix}1+pa+p^{n+1}\Z_p & 0\\0 &
1+pc+p^{n+1}\Z_p\end{pmatrix}}{a,c\in\Z_p}\quad \text{and}\\
\psi\bigl(\pr_M(\Gamma_{(g)})\bigr) &= \set{\begin{pmatrix} 1+pa +
p^{n+1}\Z_p & 0\\0 & 1+pa+p^{n+1}\Z_p\end{pmatrix}}{a\in\Z_p},
\end{align*}
whence $\lvert \psi(\Gamma_M)\rvert = p^{2n}$ and $\lvert
\psi\bigl(\pr_M(\Gamma_{(g)})\bigr)\rvert = p^n$. Therefore,
\[
\nu_M(g) = [\Gamma_M : \pr_M(\Gamma_{(g)})] =  \lvert \psi(\Gamma_M)\rvert
/\lvert \psi\bigl(\pr_M(\Gamma_{(g)})\bigr)\rvert = p^n.
\]
\end{rmk} 


\subsection{Two technical lemmas}\label{subsec:technical} 
In this subsection we prove two technical lemmas which will be needed for the
proof of the fundamental Proposition~\ref{prop:inequality}.

Recall the finite-dimensional $\R$-vector space $V$, the root system $\Phi$
inside the dual vector space $V^*$, and the set of positive roots $\Phi^+$.
Fix a subset $\Psi\subseteq \Phi^+$. We consider the partial ordering on $\Psi$
defined by
\[
\alpha\le \beta
\]
if there exist $\gamma_1,\dotsc,\gamma_n\in\Psi$ and $r,s_1,\dotsc,s_n\in
\Z_{\ge0}$, with $r>0$, such that $\beta = r\alpha + \sum_{i=1}^n s_i\gamma_i$.
The relation is clearly reflexive and transitive. It is also antisymmetric,
since there exists $v\in V$ with $\langle \alpha, v\rangle > 0$ for all
$\alpha\in \Phi^+$. 
We write $\alpha < \beta$ if $\alpha\le \beta$ and $\alpha\neq \beta$.

Let $X$ be a group. For each $\alpha\in \Psi$ let $Y_\alpha$ be a subgroup of
$X$. Define 
\[
X_\alpha \coloneqq \begin{cases}
Y_\alpha, & \text{if $2\alpha\notin \Psi$,}\\
Y_\alpha Y_{2\alpha}, & \text{if $2\alpha\in \Psi$.}
\end{cases}
\]
Put $X_{2\alpha} \coloneqq \{1\}$ if $2\alpha\notin \Psi$. We impose
the following conditions:
\begin{enumerate}[label=(\roman*)]
\item\label{rootgroup-i} $X$ is generated by the $Y_\alpha$, for $\alpha\in
\Psi$.

\item\label{rootgroup-ii} For all $\alpha,\beta\in \Psi$ the commutator subgroup
$[Y_\alpha, Y_\beta]$ is contained in the subgroup of $X$ generated by the
$Y_{r\alpha + s\beta}$, where $r,s\in \Z_{>0}$ with $r\alpha + s\beta\in \Psi$.

\item\label{rootgroup-iii} The intersection of the groups generated by
$\bigcup_{\substack{\alpha\in \Psi,\\ \langle \alpha,v\rangle\le 0}} Y_\alpha$
and $\bigcup_{\substack{\alpha\in \Psi,\\ \langle \alpha,v\rangle
>0}}Y_\alpha$, respectively, is trivial for each $v\in V$.
\end{enumerate}

Notice that $X_\alpha$ is a group thanks to \ref{rootgroup-ii}.
Fix a bijection $o\colon \Psi_{\red} \to \{1,2,\dotsc, \lvert
\Psi_{\red}\rvert\}$ and define
\[
\prod_{\alpha\in \Psi_{\red}} x_\alpha \coloneqq x_{o^{-1}(1)}\cdot
x_{o^{-1}(2)}\dotsm x_{o^{-1}(\lvert \Psi_{\red}\rvert)},\qquad \text{for
$x_\alpha\in X_\alpha$.}
\]
We call $o$ an \emph{ordering of the factors}. It follows from \cite[Lemme
(6.1.7)]{Bruhat-Tits.1972} that the multiplication map
\[
\prod_{\alpha\in \Psi_{\red}} X_\alpha \longrightarrow X
\]
is bijective. 

\begin{lem}\label{lem:aut-RG} 
Let $f\colon X\to X$ be a group homomorphism such that
\[
f(x_\alpha)x_\alpha^{-1} \in \langle X_\beta\mid \beta>\alpha\rangle,\qquad
\text{for all $x_\alpha\in X_\alpha$, all $\alpha\in \Psi_{\red}$.}
\]
For all $\prod_{\alpha\in\Psi_{\red}} x_\alpha \in X$, with $x_\alpha \in
X_\alpha$, one has
\begin{equation}\label{eq:aut-RG}
f\Bigl( \prod_{\alpha\in \Psi_{\red}} x_\alpha\Bigr) = \prod_{\alpha\in
\Psi_{\red}} z_\alpha \wt z_{\alpha}(x_\alpha) x_\alpha,
\end{equation}
where $z_\alpha = z_\alpha((x_\beta)_{\beta<\alpha}) \in X_\alpha$ depends only
on $(x_\beta)_{\beta<\alpha}$, and where
$\wt z_\alpha\colon X_\alpha\to X_{2\alpha}$ is a group homomorphism factoring
through $X_\alpha/X_{2\alpha}$.
Moreover, 
\begin{itemize}
\item $z_\alpha$ and $\wt z_\alpha$ are uniquely determined by
\eqref{eq:aut-RG}.
\item $\wt z_\alpha$ depends only on those $x_\beta$ with $\beta < \alpha$
\emph{and} $o(\beta) > o(\alpha)$.
\end{itemize}
In particular, if $o$ is such that $\beta_1 <
\beta_2$ implies $o(\beta_1) < o(\beta_2)$, then $\wt z_\alpha$
does not depend on $(x_\beta)_{\beta < \alpha}$. In this case, 
$\wt z_\alpha(x_\alpha)$ is the image of $f(x_\alpha)x_\alpha^{-1}$ under the
projection  $\prod_{\beta\in\Psi_{\red}} X_\beta \to X_{\alpha}$.
\end{lem} 
\begin{rmk} 
\begin{enumerate}[label=(\alph*)]
\item The homomorphism $f\colon X\to X$ in Lemma~\ref{lem:aut-RG} is necessarily
an automorphism.

\item The main example to keep in mind is the case where $f\colon X\to X$ is
conjugation by some element of $X$. (See Lemma~\ref{lem:inequality-1} for a
proof of why such $f$ satisfies the hypothesis of Lemma~\ref{lem:aut-RG}.) This
is also the only morphism to which we apply Lemma~\ref{lem:aut-RG}.
\end{enumerate}
\end{rmk} 
\begin{proof}[Proof of Lemma~\ref{lem:aut-RG}] 
The uniqueness assertions are immediate (for example, the uniqueness of
$z_\alpha$ follows by letting $x_\alpha = 1$).
Notice that \ref{rootgroup-ii} above implies that $X_{2\alpha}$ is central in
$X_\alpha$ and that the commutator subgroup $[X_\alpha, X_\alpha]$ is contained
in $X_{2\alpha}$.

We prove \eqref{eq:aut-RG} by induction on $\lvert \Psi_{\red}\rvert$. Suppose
$\Psi_{\red} = \{\alpha\}$. By the hypothesis on $f$ the map
\[
\wt z_\alpha\colon X_\alpha\longrightarrow X_{2\alpha},\qquad x\longmapsto
f(x)x^{-1}
\]
is well-defined and satisfies $\wt z_{\alpha}(X_{2\alpha}) = \{1\}$. Given
$x,y\in X_\alpha$, we compute
\begin{align*}
\wt z_{\alpha}(xy) &= f(xy)\cdot (xy)^{-1} = f(x)f(y) y^{-1}x^{-1}\\
&= f(x) \wt z_\alpha(y)x^{-1} = f(x)x^{-1}\cdot \wt z_\alpha(y) = \wt
z_\alpha(x)\cdot \wt z_\alpha(y).
\end{align*}
Hence, $\wt z_\alpha$ is a group homomorphism, which proves the base case (with
$z_\alpha \coloneqq 1$).

Suppose now $\lvert \Psi_{\red}\rvert > 1$ and choose a root
$\alpha_0\in \Psi_{\red}$ maximal with respect to the partial order. We start by
proving the following useful claim:

\begin{claim*} 
Suppose $f\bigl(\prod_{\alpha\in \Psi_{\red}} x_\alpha\bigr) = \prod_{\alpha \in
\Psi_{\red}} y_\alpha$, where $x_\alpha,y_\alpha\in X_\alpha$. Then
$y_{\alpha_0}$ depends only on the $x_\beta$ with $\beta\le \alpha_0$.
\begin{proof}[Proof of the claim] 
Let $\Psi'$ be the largest subset of $\Psi$ with $\Psi'_{\red}= \set{\beta\in
\Psi_{\red}}{\beta\not\le
\alpha_0}$, and let $Z_{\alpha_0}$ be the subgroup of $X$ generated by the
groups $X_\beta$, for $\beta\in \Psi'_{\red}$. Notice that $\Psi'$ is upwards
closed in $\Psi$, that is, $\gamma\ge \beta$ with $\gamma\in \Psi$ and
$\beta\in \Psi'$ implies $\gamma\in \Psi'$. Therefore, $Z_{\alpha_0}$ is a
normal subgroup of $X$. The hypotheses \ref{rootgroup-i}--\ref{rootgroup-iii}
remain satisfied if we replace $X$, $\Psi$, and $(Y_\alpha)_{\alpha\in \Psi}$ by
$Z_{\alpha_0}$, $\Psi'$, and $(Y_\alpha)_{\alpha\in \Psi'}$, respectively.
Hence, the multiplication map $\prod_{\alpha\in \Psi'_{\red}} X_\alpha \to
Z_{\alpha_0}$ is bijective, and the canonical projection
\[
\pr_{\le \alpha_0}\colon X \cong \prod_{\beta\in\Psi_{\red}} X_\beta
\longrightarrow \prod_{\beta\le \alpha_0} X_\beta \cong X/Z_{\alpha_0},
\]
is a group homomorphism with kernel $Z_{\alpha_0}$. The hypothesis on $f$
implies
$f(Z_{\alpha_0}) \subseteq Z_{\alpha_0}$, again since $\Psi'$ is upwards closed.
We obtain an induced group homomorphism $\ol f\colon X/Z_{\alpha_0} \to
X/Z_{\alpha_0}$ such that, after identifying $X\cong \prod_{\beta\in
\Psi_{\red}} X_{\beta}$ and $X/Z_{\alpha_0} \cong \prod_{\beta\le \alpha_0}
X_\beta$, the diagram
\[
\begin{tikzcd}
\prod_{\beta\in \Psi_{\red}} X_\beta \ar[r,"f"] \ar[d,"\pr_{\le \alpha_0}"'] &
\prod_{\beta\in \Psi_{\red}} X_\beta \ar[d,"\pr_{\le\alpha_0}"] \\
\prod_{\beta\le \alpha_0} X_\beta \ar[r,"\ol f"'] & \prod_{\beta\le \alpha_0}
X_\beta
\end{tikzcd}
\]
is commutative. As the restriction of $\pr_{\le\alpha_0}$ to $X_{\alpha_0}$ is
injective, it follows immediately that $y_{\alpha_0}$ only depends on the
$x_\beta$ with $\beta\le \alpha_0$. The claim is proved.
\end{proof} 
\end{claim*} 

Let $(x_\alpha)_\alpha\in \prod_{\alpha\in \Psi_{\red}}$ and write
$f\bigl(\prod_{\alpha\in\Psi_{\red}} x_\alpha\bigr) = \prod_{\alpha\in
\Psi_{\red}} y_\alpha$ as in the claim. We prove \eqref{eq:aut-RG} in two steps.
\medskip

\textit{Step~1:} We have $y_\alpha = z_\alpha \wt z_\alpha(x_\alpha)x_\alpha$,
for $\alpha\neq \alpha_0$, with $z_\alpha$ and $\wt z_\alpha$ as in the
statement of the lemma.\smallskip

This follows from the induction hypothesis as follows: as $X_{\alpha_0}$ is
normal in $X$, the quotient $X'\coloneqq X/X_{\alpha_0}$ is a group. Put $\Psi'
\coloneqq \Psi\setminus\{\alpha_0,2\alpha_0\}$. Under the projection map
$X\twoheadrightarrow X'$ the subgroups $Y_{\alpha}$, $X_{\alpha}$ of $X$ embed
into $X'$ for $\alpha\in \Psi'$. Denote $f'\colon X'\to X'$ the homomorphism
induced by $f$. The hypotheses of the lemma remain satisfied if we replace $X$,
$\Psi$, $(Y_\alpha)_{\alpha\in\Psi}$, $f$ by $X'$, $\Psi'$,
$(Y_\alpha)_{\alpha\in \Psi'}$, $f'$, respectively. The diagram
\[
\begin{tikzcd}
X \cong \prod_{\alpha\in\Psi_{\red}} X_\alpha \ar[d,"\pr"'] \ar[r,"f"] &
\prod_{\alpha\in\Psi_{\red}} X_{\alpha} \cong X \ar[d,"\pr"]\\
X' \cong \prod_{\alpha\in\Psi'_{\red}} X_\alpha \ar[r,"f'"'] &
\prod_{\alpha\in\Psi'_{\red}} X_\alpha \cong X'
\end{tikzcd}
\]
is commutative. Therefore, $f'\bigl(\prod_{\alpha\in\Psi'_{\red}} x_\alpha\bigr)
= \prod_{\alpha\in\Psi'_{\red}} y_\alpha$, and the induction hypothesis implies
$y_\alpha = z_\alpha \wt z_\alpha(x_\alpha)x_\alpha$ for certain elements
$z_\alpha\in X_\alpha$ and group homomorphisms $\wt z_\alpha\colon X_\alpha \to
X_{2\alpha}$ factoring through $X_{\alpha}/X_{2\alpha}$ and only depending on
the $x_\beta$ with $\beta < \alpha$, $\alpha\in \Psi'_{\red} = \Psi_{\red}
\setminus\{\alpha_0\}$.\medskip

\textit{Step~2:} We have $y_{\alpha_0} = z_{\alpha_0}\wt
z_{\alpha_0}(x_{\alpha_0}) x_{\alpha_0}$ with $z_{\alpha_0}$ and $\wt
z_{\alpha_0}$ as in the statement of the lemma.\smallskip

We introduce the following notation:
\begin{align*}
x^< &\coloneqq \prod_{\substack{\alpha\in\Psi_{\red}\\ o(\alpha) < o(\alpha_0)}}
x_{\alpha}, & x^> &\coloneqq \prod_{\substack{\alpha\in\Psi_{\red}\\ o(\alpha) >
o(\alpha_0)}} x_{\alpha}, & x &\coloneqq x^<\cdot x_{\alpha_0}\cdot x^>, & 
x' &\coloneqq x^<\cdot x^>,\\
f(x^<) &= \prod_{\alpha\in\Psi_{\red}} y^<_\alpha, & f(x^>) &= \prod_{\alpha\in
\Psi_{\red}} y^>_{\alpha}, & f(x) &= \prod_{\alpha\in\Psi_{\red}} y_\alpha, &
f(x') &= \prod_{\alpha\in \Psi_{\red}} y'_\alpha.
\end{align*}
The claim implies that $y^<_{\alpha_0}$ (resp. $y^>_{\alpha_0}$, resp.
$y'_{\alpha_0}$) depends only on the $x_\beta$ with $\beta <\alpha_0$ and
$o(\beta)<o(\alpha_0)$ (resp. $o(\beta) > o(\alpha_0)$, resp. $o(\beta) \neq
o(\alpha_0)$). Using that $X_{2\alpha_0}$ is central in $X$ and that
$X_{\alpha_0}$ is centralized by all $X_\beta$ with $\beta\neq \alpha_0$ we
compute
\begin{align*}
\prod_{\alpha\in\Psi_{\red}} y_\alpha &= f(x) = f(x^<)\cdot f(x_{\alpha_0})\cdot
f(x^>) = f(x^<)\cdot f(x_{\alpha_0}) \cdot \prod_{\alpha\in\Psi_{\red}}
y^>_{\alpha}\\
&= f(x^<)\cdot \Bigl(\prod_{\alpha\in\Psi_{\red}} y^>_{\alpha}\Bigr)\cdot
[f(x_{\alpha_0}), y^>_{\alpha_0}]\cdot f(x_{\alpha_0})\\
&= f(x')\cdot [f(x_{\alpha_0}), y^>_{\alpha_0}]\cdot f(x_{\alpha_0}) =
\Bigl(\prod_{\alpha\in\Psi_{\red}} y'_\alpha\Bigr)\cdot [f(x_{\alpha_0}),
y^>_{\alpha_0}]\cdot f(x_{\alpha_0}).
\end{align*}
We obtain $y_{\alpha_0} = y'_{\alpha_0}\cdot [f(x_{\alpha_0}),
y^>_{\alpha_0}]\cdot f(x_{\alpha_0})$. The claim implies that the element
$z_{\alpha_0} \coloneqq y'_{\alpha_0} \in X_{\alpha}$ only depends on the
$x_\beta$ with $\beta < \alpha_0$. Moreover, define
\begin{equation}\label{eq:aut-RG-2}
\wt z_{\alpha_0}(x_{\alpha_0}) \coloneqq [f(x_{\alpha_0}), y^>_{\alpha_0}] \cdot
f(x_{\alpha_0})x_{\alpha_0}^{-1} \in X_{2\alpha_0}.
\end{equation}
Now, $X_{2\alpha_0}$ is central in $X_{\alpha_0}$ and $f$ is the identity on
$X_{2\alpha_0}$, whence $\wt z_{\alpha_0}(X_{2\alpha_0}) = \{1\}$. It remains to
show that $\wt z_{\alpha_0}\colon X_{\alpha_0}\to X_{2\alpha_0}$ is a group
homomorphism. The base case shows that $X_{\alpha_0} \to X_{2\alpha_0}$,
$x\mapsto f(x)x^{-1}$ is a homomorphism. As $X_{2\alpha_0}$ is abelian it
suffices to show that $X_{\alpha_0} \to X_{2\alpha_0}$, $x\mapsto [f(x),
y^>_{\alpha_0}]$ is a homomorphism. But this is immediate from the general
identity $[uv,w] = u[v,w]u^{-1}\cdot [u,w]$ for all $u,v,w\in X_{\alpha_0}$.
Hence, $y_{\alpha_0} = z_{\alpha_0} \wt z_{\alpha_0}(x_{\alpha_0})
x_{\alpha_0}$, with $z_{\alpha_0}$ and $\wt z_{\alpha_0}$ depending only on the
$x_\beta$ with $\beta < \alpha_0$.\medskip

Putting together steps~1 and~2 finishes the proof of \eqref{eq:aut-RG}. For the
last statement we may assume that $\alpha = \alpha_{0}$ is maximal. Then the
claim follows from \eqref{eq:aut-RG-2}, because $y^>_{\alpha_0}$ depends only on
those $x_\beta$ with $\beta < \alpha_0$ and $o(\beta) > o(\alpha_0)$.
\end{proof} 

For the next lemma we choose the ordering $o$ of the factors such that $\alpha <
\beta$ implies $o(\alpha)<o(\beta)$.

\begin{lem}\label{lem:extra} 
Assume that $X$ is finite.
Let $Y\subseteq X$ and $Z_\alpha \subseteq X_\alpha$, for each $\alpha\in
\Psi_{\red}$, be subgroups. Assume further that the following condition is
satisfied:
\begin{equation}\label{eq:condition} 
\parbox{\dimexpr\linewidth-5em}{%
For all $1\le i\le \lvert\Psi_{\red}\rvert$ and all $(x_1,\dotsc,x_{i-1}) \in
\prod_{j=1}^{i-1} X_{o^{-1}(j)}$, there exists $z\in X_{o^{-1}(i)}$
depending only on $(x_1,\dotsc,x_{i-1})$ and satisfying the following property:
whenever $\prod_{\alpha\in \Psi_{\red}} y_\alpha \in Y$ is such that
$y_{o^{-1}(j)} = x_j$, for all $1\le j \le i-1$, we have $y_{o^{-1}(i)} \in
zZ_{o^{-1}(i)}$.
}
\end{equation} 
Then $\lvert Y \rvert \le \prod_{\alpha\in \Psi_{\red}} \lvert
Z_{\alpha}\rvert$.
\end{lem} 
\begin{proof} 
Given $0\le i\le \lvert\Psi_{\red}\rvert$ and $(x_1,\dotsc,x_{i}) \in
\prod_{j=1}^{i}X_{o^{-1}(j)}$, we put
\[
Y(x_1,\dotsc,x_{i}) \coloneqq \set{\prod_{\alpha\in \Psi_{\red}}y_\alpha \in
Y}{\text{$y_{o^{-1}(j)} = x_j$ for all $1\le j\le i$}}.
\]
For $i=0$ we define $(x_1,\dotsc,x_0)$ to be the empty tuple $()$. In this case
we have $Y() = Y$.
Whenever $i<\lvert\Psi_{\red}\rvert$ we say that $x_{i+1}\in X_{o^{-1}(i+1)}$
\emph{extends} $(x_1,\dotsc,x_{i})$ if $Y(x_1,\dotsc,x_{i+1})$ is not the empty
set. 

With $i$ as above we will prove
\begin{equation}\label{eq:extra-1}
\lvert Y(x_1,\dotsc,x_{i})\rvert \le \prod_{j=i+1}^{\lvert\Psi_{\red}\rvert}
\lvert Z_{o^{-1}(j)}\rvert,\qquad \text{for all $(x_1,\dotsc,x_{i}) \in
\prod_{j=1}^{i}X_{o^{-1}(j)}$}
\end{equation}
by descending induction on the length $i$ of the tuple $(x_1,\dotsc,x_{i})$.
Then \eqref{eq:extra-1} for $i = 0$ is precisely the assertion of the lemma.

The base case $i= \lvert\Psi_{\red}\rvert$ is satisfied, because the
right hand side equals $1$ (being the empty product) and since $\lvert
Y(x_1,\dotsc,x_{\lvert\Psi_{\red}\rvert})\rvert$ is $1$ or $0$ depending on
whether $x_1\dotsm x_{\lvert\Psi_{\red}\rvert}$ lies in $Y$ or not. Assume now
that \eqref{eq:extra-1} is satisfied for some $1\le i\le
\lvert\Psi_{\red}\rvert$. Take $(x_1,\dotsc,x_{i-1}) \in \prod_{j=1}^{i-1}
X_{o^{-1}(j)}$ and put 
\[
J(x_1,\dotsc,x_{i-1}) \coloneqq \set{y\in X_{o^{-1}(i)}}{\text{$y$ extends
$(x_1,\dotsc,x_{i-1})$}}.
\]
We may assume that $Y(x_1,\dotsc,x_{i-1})$ is non-empty. In this case
$J(x_1,\dotsc,x_{i-1})$ is also non-empty and condition~\eqref{eq:condition}
implies $\lvert J(x_1,\dotsc,x_{i-1})\rvert = \lvert Z_{o^{-1}(i)}\rvert$.
Observe that 
\[
Y(x_1,\dotsc,x_{i-1}) = \bigcup_{y\in J(x_1,\dotsc,x_{i-1})}
Y(x_1,\dotsc,x_{i-1},y).
\]
We now compute
\begin{align*}
\lvert Y(x_1,\dotsc,x_{i-1})\rvert &\le \sum_{y\in J(x_1,\dotsc,x_{i-1})} \lvert
Y(x_1,\dotsc,x_{i-1},y)\rvert\\
&\le \sum_{y\in J(x_1,\dotsc,x_{i-1})} \prod_{j=i+1}^{\lvert\Psi_{\red}\rvert}
\lvert Z_{o^{-1}(j)}\rvert\\
&= \lvert Z_{o^{-1}(i)}\rvert\cdot \prod_{j=i+1}^{\lvert\Psi_{\red}\rvert}
\lvert Z_{o^{-1}(j)}\rvert = \prod_{j=i}^{\lvert\Psi_{\red}\rvert} \lvert
Z_{o^{-1}(j)}\rvert,
\end{align*}
where the second estimate uses the induction hypothesis.
This finishes the induction step and proves the lemma.
\end{proof} 

\subsection{An inequality of indices}\label{subsec:inequality} 
Let $\Gamma\subseteq P$ be an open compact subgroup with $\Gamma = \Gamma_M
I_{U_P}$, where $\Gamma_M = \Gamma\cap M$ and $I_{U_P} = I\cap U_P$. Here $I$ is
the (fixed) Iwahori subgroup of $G$. For example, $\Gamma$ could be $K\cap P$ or
$I\cap P$ or even $I_1\cap P$.
\begin{rmk} 
Since the function $\mu_{U_P}$ takes values in $q^{\Z_{\ge0}}$, there is an
equivalence
\[
\mu_{U_P}(g) \le \mu_{U_P}(g') \iff \text{$\mu_{U_P}(g)$ divides
$\mu_{U_P}(g')$.}
\]
Therefore, we use ``$\le$'' in this context to implicitly mean ``divides''.
\end{rmk} 

The main goal of this section is to prove the following fundamental result.
\begin{prop}\label{prop:inequality} 
Each $g\in P$ with image $g_M$ in $M$ satisfies
\[
\mu_{U_P}(g) \ge \mu_{U_P}(g_M).
\]
\end{prop} 

\begin{ex} 
The ensuing proof is long and technical notwithstanding the lemmas in
subsection~\ref{subsec:technical}. We will therefore discuss first an example in
order to fix ideas. Let $G = \GL_3(\Q_p)$, $P$ the Borel subgroup of upper
triangular matrices, and $M$ the torus of diagonal matrices. Let $I$ be the
standard Iwahori determined by $P$. Let $g = g_Ug_M$ with $g_U =
\begin{psmallmatrix}1 & x & z\\0 & 1 & y\\0 & 0 & 1\end{psmallmatrix}\in U_P$
and $g_M = \diag(p^{m+n}, p^n, 1)\in M$ with $n,m\in\Z$. The inequality in
Proposition~\ref{prop:inequality} is equivalent to
\begin{equation}\label{ex:cardinality}
\big\lvert (I_{U_P})_{(g)}/H\big\rvert \le \big\lvert
(I_{U_P})_{(g_M)}/H\big\rvert,
\end{equation}
where $H \subseteq I_{U_P}$ is any sufficiently small open normal subgroup.
For a general element
$\begin{psmallmatrix}1 & u & w\\0 & 1 & v\\0 & 0 & 1\end{psmallmatrix}$ of
$I_{U_P}$ we compute
\begin{equation} \label{ex:gM} 
g_M^{-1}\begin{pmatrix}1 & u & w\\0 & 1 & v\\0 & 0 & 1\end{pmatrix} g_M =
\begin{pmatrix}1 & p^{-m}u & p^{-m-n}w\\ 0 & 1 & p^{-n}v\\ 0 & 0 &
1\end{pmatrix}.
\end{equation} 

For \eqref{ex:gM} to lie in $(I_{U_P})_{(g_M)}$ it is necessary and sufficient
that $u$, $v$, $w$ satisfy the following conditions:

\NumTabs{3}
\begin{enumerate*}[label=(\roman*), itemjoin={\tab}]
\item\label{ex:i} $\val_{\Q_p}(p^{-m}u) \ge 0$;
\item\label{ex:ii} $\val_{\Q_p}(p^{-n}v) \ge 0$;
\item\label{ex:iii} $\val_{\Q_p}(p^{-m-n}w) \ge 0$.
\end{enumerate*}

We deduce 
\begin{equation}\label{eq:IUgM}
(I_{U_P})_{(g_M)} = \begin{psmallmatrix}1 & p^{\max\{0,-m\}}\Z_p &
p^{\max\{0,-m-n\}}\Z_p\\ 0 & 1 & p^{\max\{0,-n\}}\Z_p\\ 0 & 0 &
1\end{psmallmatrix}.
\end{equation}
\medskip

Now, compute that $g^{-1}\begin{psmallmatrix}1 & u & w\\0 & 1 & v\\0 & 0 &
1\end{psmallmatrix}g$ equals
\begin{equation} \label{ex:g} 
\begin{pmatrix}
1 & p^{-m}u & y\cdot p^{-m}u - x\cdot p^{-n}v + p^{-m-n}w\\
0 & 1 & p^{-n}v\\
0 & 0 & 1
\end{pmatrix}.
\end{equation} 
Observe that \eqref{ex:g} and the right hand side of \eqref{ex:gM} differ only
in the upper right entry, and their difference is $z(u,v)\coloneqq y\cdot
p^{-m}u - x\cdot p^{-n}v$ which depends only on the terms coming from root
groups of smaller roots. 
Lemma~\ref{lem:aut-RG} shows that this is the general behavior.
For \eqref{ex:g} to lie in $(I_{U_P})_{(g)}$ it is necessary that
$u$ and $v$ satisfy \ref{ex:i}, \ref{ex:ii}, and 
\begin{enumerate}[resume, start=3, label=(\roman*')]
\item\label{ex:iii'} $\val_{\Q_p}(z(u,v)) \ge \min\{0,
-m-n\}$.
\end{enumerate}
The important observation to make here is that \ref{ex:iii'} may fail even if
$u$ and $v$ satisfy \ref{ex:i} and \ref{ex:ii}. This further restriction on $u$
and $v$ is the main reason why there is an inequality in
\eqref{ex:cardinality}.
Now, we assume that $u$ and $v$ do satisfy \ref{ex:iii'} and we determine the
possible upper right entries in \eqref{ex:g}. If $\val_{\Q_p}(z(u,v)) \ge -m-n$,
then the possible entries lie exactly in $p^{\max\{0,-m-n\}}\Z_p$. If, on the
other hand, we have $-m-n > \val_{\Q_p}(z(u,v)) \ge 0$, then they lie exactly in
the proper coset $z(u,v) + p^{\max\{0,-m-n\}}\Z_p$. Note that in both cases
$p^{\max\{0,-m-n\}}\Z_p$ does not depend on $u$ and $v$.
If $H = \begin{psmallmatrix}1 & p^r\Z_p & p^r\Z_p\\ 0 & 1 & p^r\Z_p\\ 0 & 0 &
1\end{psmallmatrix}$, for $r\gg0$ large enough, this discussion shows
\begin{align*}
\big\lvert (I_{U_P})_{(g)}/H\big\rvert &\le
\left\lvert\frac{p^{\max\{0,-m\}}\Z_p}{p^r\Z_p}\right\rvert\cdot \left\lvert
\frac{p^{\max\{0,-n\}}\Z_p}{p^r\Z_p}\right\rvert\cdot \left\lvert
\frac{p^{\max\{0,-m-n\}}\Z_p}{p^r\Z_p}\right\rvert
= \big\lvert (I_{U_P})_{(g_M)}/H\big\rvert.
\end{align*}
The role of Lemma~\ref{lem:extra} is to show that we can make this estimate in
general.
\end{ex} 
\medskip

We now turn to the proof of Proposition~\ref{prop:inequality}. As the example
above illustrates, it will be necessary to analyse $(I_{U_P})_{(g)} = I_{U_P}
\cap g^{-1}I_{U_P} g$. We will make extensive use of the identification 
\[
I_{U_P} \cong \prod_{\alpha\in\Sigma^+\setminus \Sigma_M} U_{(\alpha,0)},
\]
which follows from \eqref{eq:I1-homeo}. As $\alg G$ is not assumed
$\field$-split, the root system $\Phi$ need not be reduced. If $\alpha,
2\alpha\in \Phi$, the root group $U_{\alpha}$ is not abelian in general; it
contains the non-trivial abelian subgroup $U_{2\alpha}$, and the quotient
$U_{\alpha}/U_{2\alpha}$ is abelian. This phenomenon motivates the next
definition.

\begin{defn} 
Let $\FGrp$ be the category whose
\begin{enumerate}[label=$-$]
\item objects are pairs $(X_0,X)$ of groups such that $X_0$ is a normal subgroup
of $X$;
\item morphisms $(X_0,X) \to (X'_0,X')$ are group homomorphisms $f\colon X\to
X'$ satisfying $f(X_0) \subseteq f(X'_0)$.
\end{enumerate}
We write $(X_0,X) \subseteq (X'_0,X')$ if $X\subseteq X'$ and $X_0 = X'_0\cap
X$.
\end{defn} 

There is a canonical functor $\gr\colon \FGrp \to \Grp$ into the category of
graded groups given by
\[
\gr(X_0,X) \coloneqq \gr_0(X_0,X) \times \gr_1(X_0,X),
\]
where $\gr_0(X_0,X) \coloneqq X_0$ and $\gr_1(X_0,X) \coloneqq X/X_0$. We will
need the following elementary lemma, the proof of which will be left as an
exercise for the reader.

\begin{lem}\label{lem:gr} 
Let $(X_0,X) \subseteq (X'_0,X')$ be objects in $\FGrp$.
\begin{enumerate}[label=(\alph*)]
\item\label{lem:gr-a} One has $\gr(X_0,X) \subseteq \gr(X'_0,X')$.
\item\label{lem:gr-b} Assume $[X':X] < \infty$. Then $[X':X] = [\gr(X'_0,X') :
\gr(X_0,X)]$.
\end{enumerate}
\end{lem} 

\subsubsection{Proof of Proposition~\ref{prop:inequality}} 
Let $g\in P$. Recall that $g_M$ denotes the image of $g$ in $M$ and that
$g_U\coloneqq g_M^{-1}g \in U_P$. The group $K_M = K\cap M$ normalizes $I_{U_P}$
and hence the
function $\mu_{U_P}\colon P\to q^{\Z_{\ge0}}$ is constant on $K_MgK_M$. By
\cite[Lemma 4.1.1]{Haines-Rostami.2009} the group $K_M$ is a maximal parahoric
subgroup of $M$. Hence, the Cartan decomposition
\cite[Theorem~1.0.3]{Haines-Rostami.2009} implies that the intersection
$K_Mg_MK_M\cap Z$ is non-empty. Thus, we may assume $g_M\in Z$.

For each $\alpha\in \Sigma^+\setminus\Sigma_M$ we have $g_M^{-1}U_{(\alpha,0)}
g_M = U_{(\alpha, \langle \alpha, \nu(g_M)\rangle)}$ by \eqref{eq:U-action} and
hence
\begin{equation}\label{eq:inequality-1}
(U_{(\alpha,0)})_{(g_M)} = \begin{cases} U_{(\alpha,0)}, & \text{if $\langle
\alpha, \nu(g_M)\rangle \le 0$,}\\
U^{g_M}_{(\alpha,0)}, & \text{otherwise.}
\end{cases}
\end{equation}
Write $\Psi \coloneqq \Phi^+\setminus \Phi_M$ and choose an ordering $o$ of the
factors of $\prod_{\alpha\in \Psi_{\red}}U_\alpha$ in such a way that
$\beta<\alpha$ implies $o(\beta) < o(\alpha)$. 

Consider the automorphism $f\colon U_P\to U_P$, $x\mapsto g_U^{-1}xg_U$. We
proceed with a series of lemmas.

\begin{lem}\label{lem:inequality-1} 
For each $x_\alpha\in U_\alpha$, $\alpha\in \Psi_{\red}$, the element
$f(x_\alpha)x_\alpha^{-1}$ is contained in the subgroup $\langle U_\beta\mid
\beta >\alpha\rangle$ of $U_P$ generated by the $U_\beta$ for
$\beta\in\Psi_{\red}$ with $\beta >\alpha$.
\end{lem} 
\begin{proof} 
Write $g_U = u_{\alpha_1}\dotsm u_{\alpha_r}$ for certain $u_{\alpha_i} \in
U_{\alpha_i}$ and $\alpha_1,\dotsc,\alpha_r\in \Psi_{\red}$. We prove the
assertion by induction on $r$. 

As $(Z, (U_\alpha)_{\alpha\in\Phi})$ is a root group datum, it satisfies
\cite[(6.1.1), (DR 2)]{Bruhat-Tits.1972}:
\begin{enumerate}[label=(DR 2)]
\item\label{DR2} For each $\alpha,\beta\in\Phi$ the commutator subgroup
$[U_\alpha,U_\beta]$ is contained in the group generated by the $U_{r\alpha +
s\beta}$, for $r,s\in\Z_{>0}$ with $r\alpha+s\beta\in \Phi$.
\end{enumerate}
The base case $r=1$ is clear from \ref{DR2}. Now, assume $r>1$ and
\[
y\coloneqq (u_{\alpha_1}\dotsm u_{\alpha_{r-1}})^{-1}\cdot x_\alpha
\cdot (u_{\alpha_1}\dotsm u_{\alpha_{r-1}})\cdot x_\alpha^{-1} \in \langle
U_\beta\mid \beta >\alpha\rangle.
\]
Again from \ref{DR2} we know $y^{u_{\alpha_r}} \in \langle U_\beta \mid
\beta>\alpha\rangle$. Therefore,
$g_U^{-1}x_\alpha g_Ux_\alpha^{-1} = y^{u_{\alpha_r}}\cdot [u_{\alpha_r}^{-1},
x_\alpha]$ is contained in $\langle U_\beta\mid \beta>\alpha\rangle$, proving
the assertion.
\end{proof} 

Lemma~\ref{lem:inequality-1} allows us to apply Lemma~\ref{lem:aut-RG}: for
each $(x_\alpha)_\alpha\in \prod_{\alpha\in\Psi_{\red}} U_\alpha$ there
exist an element $z_\alpha(x_\beta)_{\beta<\alpha} \in U_\alpha$ (depending only
on $(x_\beta)_{\beta<\alpha}$) and
a group homomorphism $\wt z_\alpha\colon U_\alpha \to U_{2\alpha}$, 
factoring through $U_\alpha/U_{2\alpha}$ and independent of
$(x_\beta)_{\beta<\alpha}$, such that
\begin{equation}\label{eq:f(x)}
f\Bigl(\prod_{\alpha\in\Psi_{\red}} x_\alpha\Bigr) = \prod_{\alpha\in
\Psi_{\red}} z_\alpha(x_\beta)_{\beta<\alpha}\cdot \wt
z_{\alpha}(x_\alpha)\cdot x_\alpha.
\end{equation}
Identify $\Psi_{\red}$ with $\Sigma^+\setminus\Sigma_M$. For $\alpha\in
\Psi_{\red}$ we consider the homomorphism
\[
f_\alpha\colon U_{(\alpha,0)}^{g_M} \longrightarrow U_{\alpha},\quad x
\longmapsto \wt z_\alpha(x)\cdot x.
\]
\begin{rmk} 
Observe that $f_\alpha$ is the identity if the root system $\Phi$ is reduced
(\eg if $\alg G$ is $F$-split). In this case the next two lemmas are trivial.
\end{rmk} 

\begin{lem}\label{lem:inequality-2} 
The image of $f_\alpha$ is open in $U_\alpha$.
\end{lem} 
\begin{proof} 
Put $\Psi'\coloneqq \set{\beta\in\Psi_{\red}}{\beta\not<\alpha}$ and let
$Z_\alpha \subseteq U_P$ be the subgroup generated by the $U_\beta$ with
$\beta\in\Psi'$. Then $Z_\alpha$ is normal in $U_P$ by \ref{DR2}, and the
multiplication map induces a homeomorphism $\prod_{\beta\in \Psi'} U_\beta \cong
Z_\alpha$. The projection map $\pr_\alpha\colon Z_\alpha\to U_\alpha$ and the
automorphism $f'\coloneqq f\big|_{Z_\alpha}$, induced by the inner automorphism
$f$, are open. Hence, the subset
\[
f_\alpha\bigl(U_{(\alpha,0)}^{g_M}\bigr) = (\pr_\alpha\circ f') \Bigl(
U_{(\alpha,0)}^{g_M} \times \prod_{\beta\in \Psi'\setminus\{\alpha\}}
U_{(\beta,0)}\Bigr) \subseteq U_\alpha
\]
is open.
\end{proof} 

\begin{lem}\label{lem:inequality-3} 
$\bigl[U_{(\alpha,0)} : f_\alpha\bigl(U_{(\alpha,0)}^{g_M}\bigr)\cap
U_{(\alpha,0)}\bigr] \ge \bigl[U_{(\alpha,0)} : (U_{(\alpha,0)})_{(g_M)}\bigr]$
\end{lem} 
\begin{proof} 
For each subgroup $X$ of $U_\alpha$ we have $(U_{2\alpha}\cap X,X) \subseteq
(U_{2\alpha}, U_\alpha)$ in $\FGrp$. In order to simplify the notation we write
$X$ instead of $(U_{2\alpha}\cap X, X)$.\medskip

\textit{Step~1:} We show $\gr_0\bigl(f_{\alpha}\bigl(U_{(\alpha,0)}^{g_M}\bigr)
\cap U_{(\alpha,0)}\bigr) = \gr_0\bigl((U_{(\alpha,0)})_{(g_M)}\bigr)$, \ie
\begin{equation}\label{eq:inequality-2}
f_\alpha\bigl(U_{(\alpha,0)}^{g_M}\bigr) \cap U_{(\alpha,0)} \cap U_{2\alpha} =
(U_{(\alpha,0)})_{(g_M)} \cap U_{2\alpha}.
\end{equation}
Take $x\in U_{(\alpha,0)}$ with $f_{\alpha}(x^{g_M}) = \wt
z_{\alpha}(x^{g_M})\cdot x^{g_M} \in U_{(\alpha,0)}\cap U_{2\alpha}$. As $\wt
z_{\alpha}$ takes values in $U_{2\alpha}$, we must have $x^{g_M} \in
U_{2\alpha}$. As $\wt z_\alpha$ vanishes on $U_{2\alpha}$ we deduce
$f_\alpha(x^{g_M}) = x^{g_M}$ which is contained in the right hand side of
\eqref{eq:inequality-2}. Conversely, given $x\in U_{(\alpha,0)}$ with $x^{g_M}
\in U_{(\alpha,0)} \cap U_{2\alpha}$, we have $x^{g_M} = f_{\alpha}(x^{g_M})$,
which is contained in the left hand side of \eqref{eq:inequality-2}.\medskip

\textit{Step~2:} We prove $\gr_1\bigl(f_{\alpha}\bigl(U_{(\alpha,0)}^{g_M}\bigr)
\cap U_{(\alpha,0)}\bigr) \subseteq \gr_1\bigl((U_{(\alpha,0)})_{(g_M)}\bigr)$.
We first show
\begin{equation}\label{eq:inequality-3}
\bigl(f_\alpha\bigl(U_{(\alpha,0)}^{g_M}\bigr) \cap U_{(\alpha,0)}\bigr)\cdot
U_{2\alpha} \subseteq U_{(\alpha,0)}^{g_M}U_{2\alpha} \cap
U_{(\alpha,0)}U_{2\alpha} = (U_{(\alpha,0)})_{(g_M)}\cdot U_{2\alpha}.
\end{equation}
The inclusion is a consequence of
$f_\alpha\bigl(U_{(\alpha,0)}^{g_M}\bigr)U_{2\alpha} = U_{(\alpha,0)}^{g_M}
U_{2\alpha}$ and the equality follows from~\eqref{eq:inequality-1}. We compute
\begin{align*}
\gr_1\bigl(f_{\alpha}\bigl(U_{(\alpha,0)}^{g_M}\bigr) \cap U_{(\alpha,0)}\bigr)
&= \frac{f_{\alpha}\bigl(U_{(\alpha,0)}^{g_M}\bigr) \cap
U_{(\alpha,0)}}{f_{\alpha}\bigl(U_{(\alpha,0)}^{g_M}\bigr) \cap U_{(\alpha,0)}
\cap U_{2\alpha}}
= \frac{\bigl(f_{\alpha}\bigl(U_{(\alpha,0)}^{g_M}\bigr) \cap
U_{(\alpha,0)}\bigr)\cdot U_{2\alpha}}{U_{2\alpha}}\\
&\stackrel{\eqref{eq:inequality-3}}{\subseteq}
\frac{(U_{(\alpha,0)})_{(g_M)}\cdot U_{2\alpha}}{U_{2\alpha}}
= \frac{(U_{(\alpha,0)})_{(g_M)}}{(U_{(\alpha,0)})_{(g_M)}\cap U_{2\alpha}}
= \gr_1\bigl((U_{(\alpha,0)})_{(g_M)}\bigr).
\end{align*}

\textit{Step~3:} Proof of the assertion. Steps~1 and~2 imply
$\gr\bigl(f_\alpha\bigl(U_{(\alpha,0)}^{g_M}\bigr) \cap U_{(\alpha,0)}\bigr)
\subseteq \gr \bigl((U_{(\alpha,0)})_{(g_M)}\bigr)$. The index of
$f_\alpha\bigl(U_{(\alpha,0)}^{g_M}\bigr) \cap U_{(\alpha,0)}$ in
$U_{(\alpha,0)}$ is finite by Lemma~\ref{lem:inequality-2}. Applying
Lemma~\ref{lem:gr}.\ref{lem:gr-b} twice finally shows
\begin{align*}
\bigl[U_{(\alpha,0)} : f_{\alpha}\bigl(U_{(\alpha,0)}^{g_M}\bigr) \cap
U_{(\alpha,0)}\bigr] &= \bigl[\gr(U_{(\alpha,0)}) : \gr\bigl(
f_\alpha\bigl(U_{(\alpha,0)}^{g_M}\bigr) \cap U_{(\alpha,0)}\bigr)\bigr]\\
&\ge \bigl[\gr(U_{(\alpha,0)}) :
\gr\bigl((U_{(\alpha,0)})_{(g_M)}\bigr)\bigr]\\
&= \bigl[U_{(\alpha,0)} : (U_{(\alpha,0)})_{(g_M)}\bigr]. \qedhere
\end{align*}
\end{proof} 

Given $\alpha\in \Psi_{\red}$ and $(x_\beta)_{\beta< \alpha} \in
\prod_{\substack{\beta\in \Psi_\red\\ \beta < \alpha}} U_\beta$, we consider the
subset
\[
X_{(x_\beta)_{\beta<\alpha}} \coloneqq U_{(\alpha,0)} \cap
\set{z_\alpha\bigl(x_\beta^{g_M}\bigr)_{\beta<\alpha}\cdot \wt
z_\alpha\bigl(x_\alpha^{g_M}\bigr)\cdot x_\alpha^{g_M}}{x_\alpha \in
U_{(\alpha,0)}} \subseteq U_{(\alpha,0)}.
\]
As $(I_{U_P})_{(g)} = I_{U_P}\cap f\bigl(I_{U_P}^{g_M}\bigr)$ it is immediate
from \eqref{eq:f(x)} that
\begin{equation}\label{eq:inequality-4}
(I_{U_P})_{(g)} = \set{\prod_{\alpha\in \Psi_{\red}} y_\alpha = f\Bigl(
\prod_{\alpha\in\Psi_{\red}} x_\alpha^{g_M}\Bigr) \in \prod_{\alpha\in
\Psi_{\red}} U_{(\alpha,0)}}{\begin{array}{l} 
\text{$x_\alpha\in U_{(\alpha,0)}$ and}\\
\text{$y_\alpha\in X_{(x_\beta)_{\beta<\alpha}}$}\\
\text{for all $\alpha\in\Psi_{\red}$}
\end{array}}.
\end{equation}

\begin{lem}\label{lem:inequality-4} 
The set $X_{(x_\beta)_{\beta<\alpha}}$ is either empty or a left coset of
$f_\alpha\bigl(U_{(\alpha,0)}^{g_M}\bigr) \cap U_{(\alpha,0)}$.
\end{lem} 
\begin{proof} 
Clearly, $X_{(x_\beta)_{\beta<\alpha}}$ is stable under right multiplication by
elements of $f_\alpha\bigl(U_{(\alpha,0)}^{g_M}\bigr) \cap U_{(\alpha,0)}$. Now,
take two elements
\[
\gamma_i\coloneqq z_\alpha\bigl(x_\beta^{g_M}\bigr)_{\beta<\alpha} \cdot \wt
z_\alpha(x_i^{g_M}) \cdot x_i^{g_M} \in X_{(x_\beta)_{\beta<\alpha}},\quad
\text{with $x_i \in U_{(\alpha,0)}$, $i=1,2$.}
\]
Then $\gamma_2^{-1}\gamma_1 = f_\alpha\bigl((x_2^{-1}x_1)^{g_M}\bigr) \in
f_\alpha\bigl(U_{(\alpha,0)}^{g_M}\bigr) \cap U_{(\alpha,0)}$ proving the claim.
\end{proof} 

Choose $r\in\Z_{>0}$ big enough so that $U_{(\alpha,r)}$ is contained in
$f_\alpha\bigl(U_{(\alpha,0)}^{g_M}\bigr) \cap U_{(\alpha,0)} \cap
U_{(\alpha,0)}^{g_M}$ for all $\alpha\in \Psi_\red$ and such that the subgroup
$H\coloneqq \prod_{\alpha\in\Psi_{\red}} U_{(\alpha,r)}$ of $I_{U_P}$
is contained in $(I_{U_P})_{(g)} \cap (I_{U_P})_{(g_M)}$. Observe that $H$ is
normal, since the valuation $\varphi_0$ satisfies \cite[(6.2.1), 
(V 3)]{Bruhat-Tits.1972}. 

We are going to apply Lemma~\ref{lem:extra} next with $X$, $Y$, $Z_\alpha$ equal
to $I_{U_P}/H$ and $(I_{U_P})_{(g)}/H$ and $f_\alpha(U_{(\alpha,0)}^{g_M}) \cap
U_{(\alpha,0)}/U_{(\alpha,r)}$, respectively. If $\prod_{\alpha\in
\Psi_{\red}} x_\alpha \in \prod_{\alpha\in\Psi_\red}U_{(\alpha,0)}$ and
$\prod_{\alpha\in \Psi_{\red}} y_\alpha = f(\prod_{\alpha\in\Psi_\red}
x_\alpha^{g_M})$, then each $x_\alpha$, and hence $(x_\beta)_{\beta <\alpha}$,
depends only on $(y_\beta)_{\beta <\alpha}$ (apply the Lemma~\ref{lem:aut-RG}
to $f^{-1}$). Therefore, $X_{(x_\beta)_{\beta < \alpha}}$ depends only on
$(y_\beta)_{\beta <\alpha}$ and, in particular, on
$(y_1,\dotsc,y_{o(\alpha)-1})$ by our choice of $o$.
In fact, $X_{(x_\beta)_{\beta<\alpha}}$ depends only on the cosets $y_\beta
U_{(\beta,r)}$, for $\beta < \alpha$, since we assumed
$U_{(\beta,r)} \subseteq f_\beta\bigl(U_{(\beta,0)}^{g_M}\bigr) \cap
U_{(\beta,0)}$.
Now, \eqref{eq:inequality-4} and Lemma~\ref{lem:inequality-4} show that
condition~\eqref{eq:condition} is satisfied. Lemma~\ref{lem:extra} implies
\begin{equation}\label{eq:inequality-5}
\big\lvert (I_{U_P})_{(g)}/H\big\rvert \le \prod_{\alpha\in \Psi_{\red}}
\big\lvert f_\alpha(U_{(\alpha,0)}^{g_M}) \cap U_{(\alpha,0)} / U_{(\alpha,r)}
\big\rvert.
\end{equation}

Further, using Lemma~\ref{lem:inequality-3} we estimate
\begin{align}\label{eq:inequality-6}
\big\lvert f_\alpha\bigl(U_{(\alpha,0)}^{g_M}\bigr) \cap U_{(\alpha,0)}
/U_{(\alpha,r)}\big\rvert &= \frac{\lvert U_{(\alpha,0)}/U_{(\alpha,r)}\rvert}
{[U_{(\alpha,0)} : f_\alpha\bigl(U_{(\alpha,0)}^{g_M}\bigr) \cap
U_{(\alpha,0)}]}\\
&\le \frac{\lvert U_{(\alpha,0)} /U_{(\alpha,r)}\rvert}{[U_{(\alpha,0)} :
(U_{(\alpha,0)})_{(g_M)}]} = \big\lvert (U_{(\alpha,0)})_{(g_M)} /
U_{(\alpha,r)}\big\rvert. \notag
\end{align}

Putting \eqref{eq:inequality-5} and \eqref{eq:inequality-6} together yields
\[
\big\lvert (I_{U_P})_{(g)}/H\big\rvert \le \prod_{\alpha\in \Psi_\red}
\big\lvert (U_{(\alpha,0)})_{(g_M)}/U_{(\alpha,r)}\big\rvert = \big\lvert
(I_{U_P})_{(g_M)}/H\big\rvert.
\]

Finally, we conclude $\mu_{U_P}(g) = \frac{\lvert I_{U_P}/H\rvert}
{\lvert(I_{U_P})_{(g)}/H\rvert} \ge \frac{\lvert I_{U_P}/H\rvert} {\lvert
(I_{U_P})_{(g_M)}/H\rvert} = \mu_{U_P}(g_M)$, finishing the proof of
Proposition~\ref{prop:inequality}.

\subsection{Properties of \texorpdfstring{$\mu_{U_P}(w)$}{mu(w)}} 
\label{subsec:properties}
\begin{nota} 
As $\mu_{U_P}\colon M\to q^{\Z_{\ge0}}$ is constant on double cosets with
respect to $K_M$ (hence also $I_M$), we obtain
from the Bruhat decomposition~\eqref{eq:Bruhat-Decomposition} an induced function
\begin{equation}\label{eq:mu-WM}
\mu_{U_P}\colon W_M \longrightarrow q^{\Z_{\ge0}}.
\end{equation}
As $K_M$ contains representatives of $W_{0,M}$, it follows that $\mu_{U_P}$ is
constant on the double cosets with respect to $W_{0,M}$. 

The analogous map
$\mu_{U_P}\colon W_M(1)\to q^{\Z_{\ge0}}$ is obtained from \eqref{eq:mu-WM} by
inflation. It is clear that all results in this section are still true if we
replace $W_M$ by $W_M(1)$. 
\end{nota} 
Our goal in this section will be to study the properties of the function
$\mu_{U_P}$.

\begin{lem}\label{lem:mu-prod} 
Let $\lambda \in \Lambda$ and $w_0\in W_{0,M}$. Then
\[
\mu_{U_P}(e^\lambda w_0) = \prod_{\substack{\alpha\in \Sigma^+\setminus \Sigma_M\\
\langle \alpha,\nu(\lambda )\rangle > 0}} \big\lvert U_{(\alpha,0)}/U_{(\alpha, \langle
\alpha,\nu(\lambda )\rangle)}\big\rvert.
\]
\end{lem} 
\begin{proof} 
Let $m\in Z$ be a representative of $\lambda \in \Lambda$. Then
the multiplication map induces an $m$-equivariant bijection
$I_{U_P} \cong \prod_{\alpha\in\Sigma^+\setminus\Sigma_M} U_{(\alpha,0)}$, and
we compute
\begin{align*}
(I_{U_P})_{(m)} = I_{U_P}\cap m^{-1}I_{U_P}m &\cong
\prod_{\alpha\in\Sigma^+\setminus\Sigma_M}
U_{(\alpha,0)} \cap m^{-1}U_{(\alpha,0)}m\\
&= \prod_{\substack{\alpha\in \Sigma^+\setminus\Sigma_M\\ \langle
\alpha,\nu(\lambda )\rangle \le 0}} U_{(\alpha,0)} \times \prod_{\substack{\alpha \in
\Sigma^+\setminus\Sigma_M\\ \langle\alpha,\nu(\lambda )\rangle >0}} U_{(\alpha, \langle
\alpha,\nu(\lambda )\rangle)}.
\end{align*}
But then
\[
\mu_{U_P}(e^\lambda w_0) = \mu_{U_P}(e^\lambda ) = [I_{U_P} : (I_{U_P})_{(m)}] = \prod_{\substack{\alpha \in
\Sigma^+\setminus \Sigma_M\\ \langle \alpha,\nu(\lambda )\rangle >0}} \big\lvert
U_{(\alpha,0)}/U_{(\alpha,\langle\alpha,\nu(\lambda )\rangle)}\big\rvert. \qedhere
\]
\end{proof} 

\begin{defn} 
Given $(\alpha,k) \in \Sigma^\aff =\Sigma\times\Z$, we put
(cf.~\eqref{eq:q-hyperplanes})
\[
q(\alpha,k) \coloneqq q(H_{(\alpha,k)}) = \big\lvert
U_{(\alpha,k)}/U_{(\alpha,k+1)}\big\rvert \in
q^{\Z_{>0}}.
\]
Notice that $q(\alpha,k) = q\bigl(w\cdot (\alpha,k)\bigr)$ by
\eqref{eq:action-U}. In particular, $q(\alpha,k) = q(-\alpha,k)$ (take $w =
s_{\alpha}$) and $q(\alpha,k) = q(\alpha,k+\langle \alpha,\nu(\lambda)\rangle)$
(take $w = e^{-\lambda}\in \Lambda$).
\end{defn} 

\subsubsection{The opposite parabolic} 
Let $\alg P^\op$ be the parabolic subgroup opposite $\alg P$ with Levi $\alg M$
and unipotent radical $\alg U_{\alg P^\op}$. Writing $I_{U_{P^\op}} =
I\cap U_{P^\op}$, we define
\[
\mu_{U_{P^\op}}(g) \coloneqq [I_{U_{P^\op}} : (I_{U_{P^\op}})_{(g)}],\qquad
\text{for $g\in P$.}
\]

It is constant on double cosets with respect to $K_M$, hence restricts to a
function $\mu_{U_{P^\op}}\colon W_M\to q^{\Z_{\ge0}}$. The following proposition
explains how $\mu_{U_{P^\op}}$ is related with $\mu_{U_P}$.

\begin{prop}\label{prop:mu-P^op} 
$\mu_{U_{P^\op}}(w) = \mu_{U_P}(w^{-1})$ for all $w\in W_M$.
\end{prop} 
\begin{proof} 
We may assume $w = e^\lambda\in \Lambda$. Lemma~\ref{lem:mu-prod} implies
\[
\mu_{U_P}(-\lambda) = \prod_{\substack{\alpha\in\Sigma^+\setminus \Sigma_M\\
\langle\alpha, \nu(-\lambda)\rangle>0}} \prod_{k=0}^{\langle
\alpha,\nu(-\lambda)\rangle - 1} q(\alpha,k).
\]
A similar argument shows
\[
\mu_{U_{P^\op}}(\lambda) = 
\prod_{\substack{\alpha\in (-\Sigma^+)\setminus\Sigma_M\\ \langle
\alpha,\nu(\lambda)\rangle >0}}
\big\lvert U_{(\alpha,1)}/U_{(\alpha,\langle\alpha,\nu(\lambda)\rangle +1)}
\big\rvert = 
\prod_{\substack{\alpha\in (-\Sigma^+)\setminus\Sigma_M\\ \langle
\alpha,\nu(\lambda)\rangle >0}}
\prod_{k=1}^{\langle\alpha,\nu(\lambda)\rangle} q(\alpha,k).
\]
Using $\langle -\alpha,\nu(-\lambda)\rangle = \langle
\alpha,\nu(\lambda)\rangle$, $q(-\alpha,k) = q(\alpha,k)$, and $q(\alpha,0) =
q(\alpha,\langle\alpha,\nu(\lambda)\rangle)$, we obtain $\mu_{U_P}(-\lambda) =
\mu_{U_{P^\op}}(\lambda)$.
\end{proof} 

\subsubsection{Changing the parabolic subgroup} 
Let $\alg Q = \alg L\alg U_{\alg Q}$ be a parabolic subgroup of $\alg G$ with
$\alg M\subseteq \alg L$. Then $\alg P\cap \alg L$ is a parabolic subgroup of
$\alg L$ with Levi $\alg M$ and unipotent radical $\alg U_{\alg P}\cap \alg L$.

\begin{prop}\label{prop:changing} 
$\mu_{U_P}(w) = \mu_{U_P\cap L}(w)\cdot \mu_{U_Q}(w)$ for all $w\in W_M$.
\end{prop} 
\begin{proof} 
It suffices to prove the assertion for $w\in \Lambda$. Using
Lemma~\ref{lem:mu-prod} we obtain
\begin{align*}
\mu_{U_P}(w) &= \prod_{\substack{\alpha\in \Sigma^+\setminus\Sigma_M\\
\langle\alpha,\nu(w)\rangle>0}} \big\lvert U_{(\alpha,0)}/U_{(\alpha,\langle
\alpha,\nu(w)\rangle)}\big\rvert\\
&= \prod_{\substack{\alpha\in \Sigma_L^+\setminus \Sigma_M\\ \langle\alpha,
\nu(w)\rangle>0}} \big\lvert U_{(\alpha,0)}/U_{(\alpha,\langle\alpha,
\nu(w)\rangle)}\big\rvert\cdot
\prod_{\substack{\alpha\in \Sigma^+\setminus\Sigma_L\\ \langle\alpha,\nu(w)
\rangle>0}} \big\lvert U_{(\alpha,0)}/U_{(\alpha,\langle\alpha,\nu(w)\rangle)}
\big\rvert\\
&= \mu_{U_P\cap L}(w)\cdot \mu_{U_Q}(w).\qedhere
\end{align*}
\end{proof} 

\subsubsection{Relating \texorpdfstring{$\mu_{U_P}$}{mu} with the 
structure of \texorpdfstring{$W_M$}{WM}}

The fact that $\mu_{U_P}$ can be defined on $W_M$ may seem rather coincidental.
There is, however, a strong relationship between the two. To begin with, the
function
\begin{equation}\label{eq:deltaM}
\delta_M\colon M \longrightarrow \Q^\times,\qquad m\longmapsto
\mu_{U_P}(m)/\mu_{U_P}(m^{-1})
\end{equation}
is easily seen to be a group homomorphism by showing $\delta_M(m) =
[mI_{U_P}m^{-1}: I_{U_P}]$ (generalized index, see \cite[I.2.7]{Vigneras.1996}
for various properties), and this clearly induces a group homomorphism
$\delta_M\colon W_M\to \Q^\times$. One can say more:
\begin{lem}\label{lem:deltaM} 
$\delta_M$ is trivial on $W^\aff$, hence factors through a character
$\Omega_M\to \Q^\times$.
\end{lem} 
\begin{proof} 
Given $s\in S^\aff$, we have $s = s^{-1}$. As $\delta_M$ is a group homomorphism
taking only positive values, this implies $\delta_M(s) = 1$. As $W^\aff$ is
generated by $S^\aff$, the assertion follows.
\end{proof} 

Thus, $\mu_{U_P}$ carries (at least some) information about the group structure
of $W_M$. Our main result in this section shows that $\mu_{U_P}$ measures the
deviation between the length functions $\ell$ on $W$ and $\ell_M$ on $W_M$, and
that it is monotone with respect to the Bruhat order $\le_M$ on $W_M$. 

\begin{prop}\label{prop:mu-prop} 
Let $v,w\in W_M$. Then:
\begin{enumerate}[label=(\alph*)]
\item\label{prop:mu-prop-a} $q_w = \mu_{U_P}(w)\mu_{U_P}(w^{-1})\cdot q_{M,w}$;
\item\label{prop:mu-prop-b} $\mu_{U_P}(v) \le \mu_{U_P}(w)$ whenever $v\le_M w$.
\end{enumerate}
\end{prop} 

Before we give the proof, we deduce a simple corollary:
\begin{cor}\label{cor:mu-prop} 
$\mu_{U_P}(vw)\le \mu_{U_P}(v)\mu_{U_P}(w)$ and $q_{v,w}
= \frac{\mu_{U_P}(v)\mu_{U_P}(w)}{\mu_{U_P}(vw)}\cdot q_{M,v,w}$ for all $v,w\in
W_M$.
\end{cor} 
\begin{proof} 
As $\delta_M\colon W_M\to \Q^\times$ is a group homomorphism,
we have $\frac{\mu_{U_P}(v)\mu_{U_P}(w)}{\mu_{U_P}(vw)} =
\frac{\mu_{U_P}(v^{-1})\mu_{U_P}(w^{-1})}{\mu_{U_P}(w^{-1}v^{-1})}$. Hence,
Proposition~\ref{prop:mu-prop}.\ref{prop:mu-prop-a} implies
\begin{align*}
q_{v,w}^2 &= \frac{q_vq_w}{q_{vw}} = \frac{\mu_{U_P}(v)\mu_{U_P}(v^{-1})\cdot
\mu_{U_P}(w)\mu_{U_P}(w^{-1})}{\mu_{U_P}(vw)\mu_{U_P}(w^{-1}v^{-1})}\cdot
\frac{q_{M,v}\cdot q_{M,w}}{q_{M,vw}}\\
&= \left(\frac{\mu_{U_P}(v)\cdot \mu_{U_P}(w)}{\mu_{U_P}(vw)}\cdot
q_{M,v,w}\right)^2.
\end{align*}
Taking the (positive) square root implies the second assertion. Since
$\hyperplanes_{M,w} = \hyperplanes_w\cap \hyperplanes_{M}$ for all $w\in W$
(see~\eqref{eq:qw-hyperplanes} for the definition of $\hyperplanes_w$), we
have $\hyperplanes_{M,v}\cap v\hyperplanes_{M,w} \subseteq \hyperplanes_v \cap
v\hyperplanes_w$. Therefore $q_{M,v,w}$ divides $q_{v,w}$, and the first
assertion follows from the second.
\end{proof} 

\medskip

\begin{proof}[Proof of Proposition~\ref{prop:mu-prop}] 
\begin{itemize}
\item[\ref{prop:mu-prop-a}] Under the inclusion $\hyperplanes_M \subseteq
\hyperplanes$ we have
\[
\hyperplanes_M = \set{H_{(\alpha,k)} \in \hyperplanes}{(\alpha,k) \in
\Sigma_M^\aff}.
\]
Take $w = e^\lambda w_0$, with $\lambda\in \Lambda$ and $w_0\in W_{0,M}$. 
By \cite[Lemma~5.6 and Proposition~5.9, 2)]{Vigneras.2016} (using
$w_0(\Sigma^+\setminus\Sigma_M) = \Sigma^+\setminus\Sigma_M$) there is a
decomposition
\[
\hyperplanes_w = \hyperplanes_w^+ \sqcup \hyperplanes_w^- \sqcup
\hyperplanes_{M,w},
\]
where
\begin{align*}
\hyperplanes_w^+ &\coloneqq \set{H_{(\alpha,k)}}{\begin{array}{l} \text{$\alpha
\in \Sigma^+\setminus\Sigma_M$ with $\langle \alpha,\nu(\lambda)\rangle > 0$}\\
\text{and $k\in \{-\langle \alpha,\nu(\lambda)\rangle,\dotsc,-1\}$}
\end{array}}\qquad \text{and}\\
\hyperplanes_w^- &\coloneqq \set{H_{(\alpha,k)}}{\begin{array}{l}
\text{$\alpha\in \Sigma^+\setminus\Sigma_M$ with $\langle \alpha,
\nu(\lambda)\rangle <0$}\\
\text{and $k\in \{0,\dotsc,-\langle\alpha,\nu(\lambda)\rangle -1\}$}
\end{array}}.
\end{align*}
Thanks to Lemma~\ref{lem:mu-prod} we compute
\[
\prod_{H\in \hyperplanes_w^-}q(H) = \prod_{\substack{\alpha\in \Sigma^+\setminus
\Sigma_M\\ \langle\alpha,\nu(\lambda)\rangle<0}} \prod_{k=0}^{-\langle
\alpha,\nu(\lambda)\rangle - 1} q(\alpha,k) = \mu_{U_P}(-\lambda) =
\mu_{U_P}(w^{-1}).
\]
Likewise, a short computation reveals $\prod_{H\in \hyperplanes_w^+}q(H) =
\mu_{U_P}(w)$ keeping in mind $q(\alpha,k) = q(\alpha,
k+\langle\alpha,\nu(\lambda)\rangle)$. Taking everything together, the assertion
follows from \eqref{eq:qw-hyperplanes} applied to both $q_w$ and $q_{M,w}$.

\item[\ref{prop:mu-prop-b}] Let $v,w\in W_M$ with $v\le_Mw$. By definition of
the Bruhat order, \eg \cite[Definition~2.1.1]{Bjorner-Brenti.2006}, there
exists a chain $v = v_0, v_1,\dotsc,v_k = w$ in $W_M$ with $v_iv_{i-1}^{-1} \in
S(\hyperplanes_M)$ ($\iff v_{i-1}^{-1}v_i \in S(\hyperplanes_M)$) and
$\ell_M(v_{i-1}) < \ell_M(v_i)$ for all $1\le i\le k$.

We are therefore reduced to the case $w = tv$ and $\ell_M(v) < \ell_M(w)$ for
some $t\in S(\hyperplanes_M)$. In fact, by
\cite[Theorem~2.2.6]{Bjorner-Brenti.2006}, we may even assume $\ell_M(w) =
\ell_M(v) + 1$. 
\medskip

\textit{Step~1:} We show $\mu_{U_P}(tw)\mu_{U_P}\bigl((tw)^{-1}\bigr) \le
\mu_{U_P}(w)\mu_{U_P}(w^{-1})$. 
The hypothesis $\ell_M(tw) = \ell_M(w) - 1$ implies $H_t \in \hyperplanes_{M,w}
\setminus \hyperplanes_{M,tw}$. But since $\hyperplanes_{M,w'} =
\hyperplanes_{w'}\cap \hyperplanes_M$, for all $w'\in W_M$, we also have $H_t\in
\hyperplanes_w \setminus \hyperplanes_{tw}$. Let $\Gamma_w$ be a minimal gallery
in $\apartment$ connecting $\chamber$ and $w\chamber$. Denote $(H_1,H_2,\dotsc,
H_{\ell(w)})$ the sequence of hyperplanes crossed by $\Gamma_w$. Then $H_t =
H_i$ for some $1\le i\le \ell(w)$. By folding $\Gamma_w$ along $H_t$ and
deleting the repeated chamber we obtain a gallery $\Gamma_{tw}$ of length
$\ell(w) - 1$ in $\apartment$ connecting $\chamber$ and $tw\chamber$, given by
crossing the hyperplanes
\[
(H_1,\dotsc,H_{i-1}, tH_{i+1},\dotsc,tH_{\ell(w)}).
\]
Of course, $\Gamma_{tw}$ crosses all hyperplanes in $\hyperplanes_{tw}$ at least
once. Applying \ref{prop:mu-prop-a} twice, we compute
\begin{align*}
\mu_{U_P}(w)\mu_{U_P}(w^{-1})q_{M,w} = q_w &= \prod_{j=1}^{\ell(w)} q(H_j)\\
&= \prod_{j=1}^{i-1}q(H_j)\cdot \prod_{j=i+1}^{\ell(w)} q(tH_j)\cdot q(H_t)\\
&\ge q_{tw}\cdot q(H_t)\\
&= \mu_{U_P}(tw)\mu_{U_P}\bigl((tw)^{-1}\bigr)\cdot q_{M,tw}\cdot q(H_t).
\end{align*}
Since $\ell_M(w) = \ell_M(tw)+1$, we have $q_{M,w} = q_{M,tw}\cdot q(H_t)$,
finishing the proof of step~1.
\medskip

\textit{Step~2:} We show $\mu_{U_P}(tw) \le \mu_{U_P}(w)$. By
Lemma~\ref{lem:deltaM} we have $\delta_M(tw) = \delta_M(w)$. Hence,
\begin{align*}
\mu_{U_P}(tw)^2 &= \mu_{U_P}(tw)\mu_{U_P}\bigl((tw)^{-1}\bigr)\cdot
\delta_M(tw)\\
&\le \mu_{U_P}(w)\mu_{U_P}(w^{-1})\cdot \delta_M(w) = \mu_{U_P}(w)^2
\end{align*}
using step~1. Taking square roots, the assertion follows.\qedhere
\end{itemize}
\end{proof} 
\subsection{Positive elements}\label{subsec:positive} 
We collect here some results on positive elements in the Levi subgroup $M$ which
will become important for the definition of parabolic induction. The
results in this section are not new, but we will give new proofs of two results
obtained by Abe in \cite{Abe.2016a}

\begin{defn} 
An element $m\in M$ is called \emph{$M$-positive} (or just \emph{positive}
if no confusion arises) if $\mu_{U_P}(m) = 1$. The monoid
of positive elements is denoted $M^+$ (or even $M^{+,G}$ if we want to stress
that $M$ is considered as a Levi subgroup in $G$). Notice that $K_M\subseteq
M^+$. 

The elements in the monoid $M^- \coloneqq (M^+)^{-1}$ are called
\emph{$M$-negative} (or \emph{negative}).
\end{defn} 

\begin{rmk} 
In view of $\mu_{U_{P^\op}}(m^{-1}) = \mu_{U_P}(m)$
(Proposition~\ref{prop:mu-P^op}), $m$ being positive is equivalent to
\[
mI_{U_P}m^{-1} \subseteq I_{U_P}\quad \text{and}\quad I_{U_{P^\op}} \subseteq
mI_{U_{P^\op}} m^{-1},
\]
where $I_{U_{P^\op}} = I\cap U_{P^\op}$. This recovers the classical definition,
cf. \cite[(6.5)]{Bushnell-Kutzko.1998} or \cite[II.4]{Vigneras.1998}.
\end{rmk} 

\begin{defn} 
\begin{enumerate}[label=(\alph*)]
\item An element $w\in W_M$ (or in $W_M(1)$) is called \emph{$M$-positive} (or
\emph{positive}) if $\mu_{U_P}(w) = 1$. The monoid of positive elements in $W_M$ is
denoted $W_{M^+}$.

The elements in $W_{M^-} \coloneqq (W_{M^+})^{-1}$ are called \emph{$M$-negative} (or
\emph{negative}).

\item An element $\lambda\in \Lambda$ (or in $\Lambda(1)$) is called
\emph{strictly $M$-positive} (or \emph{strictly positive}) if there exists a
central element $a\in M$ with $\lambda = aZ_0$ (or $\lambda = aZ_1$) and if
\[
\langle \alpha, \nu(\lambda)\rangle < 0\quad \text{for all $\alpha \in
\Sigma^+\setminus\Sigma_M$.}
\]
\item A central element $a\in M$ is called \emph{strictly $M$-positive} (or
\emph{strictly positive}) if $aZ_0\in \Lambda$ is.
\end{enumerate}
\end{defn} 

\begin{rmk} 
\begin{enumerate}[label=(\alph*)]
\item If $\lambda\in \Lambda$ is strictly positive, then $\langle\alpha,
\nu(\lambda)\rangle = 0$ for all $\alpha\in \Sigma_M$, because $\lambda = aZ_0$
for some element $a$ in the center of $M$. The strictly positive elements are
contained in $\Lambda_{M^+} \coloneqq \Lambda \cap W_{M^+}$ (\eg by
Lemma~\ref{lem:mu-prod}).

\item An element $\lambda\in \Lambda$ is strictly positive if and only if it
lies in the image of a strongly positive element $a$ in the sense of
\cite[(6.16)]{Bushnell-Kutzko.1998}. In particular,
\cite[(6.14)]{Bushnell-Kutzko.1998}, strictly positive elements exist.
Moreover, given any $m\in M$ we have $a^nm \in M^+$ for $n\gg0$.

\item The Bruhat decomposition of $M$ \eqref{eq:Bruhat-Decomposition} induces
bijections $W_{M^+}\cong I_M\backslash M^+/I_M$ and $W_{M^+}(1) \cong I_{1,M}
\backslash M^+/I_{1,M}$ \cite[Remark~2.11 (2)]{Ollivier-Vigneras.2018}.

\item We have $W_{M^+} \cong \Lambda_{M^+} \rtimes W_{0,M}$ and
\cite[Lemma~2.2]{Vigneras.2015}
\[
\Lambda_{M^+} = \set{\lambda\in \Lambda}{\langle \alpha,\nu(\lambda)\rangle \le
0\quad \text{for all $\alpha\in \Sigma^+\setminus\Sigma_M$}}.
\]
It follows that if $\lambda\in \Lambda$ (or in $\Lambda(1)$) is strictly
positive, then for all $w\in W_M$ (or  in $W_M(1)$) there exists $n\gg0$ with $e^{n\lambda}w\in
W_{M^+}$ (or in $W_{M^+}(1)$).
\end{enumerate}
\end{rmk} 

Finally, we give new proofs of two useful lemmas due to Abe. They will not be
needed in the rest of the paper.

\begin{cor}[{\cite[Lemma~4.5]{Abe.2016a}}] 
\label{cor:Abe-1}
Let either $v,w\in W_{M^+}$ or $v,w \in W_{M^-}$. Then we have $q_{v,w} =
q_{M,v,w}$ and, in particular,
\[
\ell_M(v) + \ell_M(w) - \ell_M(vw) = \ell(v) + \ell(w) - \ell(vw).
\]
\end{cor} 
\begin{proof} 
If $v,w\in W_{M^+}$ then Corollary~\ref{cor:mu-prop} yields $q_{v,w} =
q_{M,v,w}$ (and also $vw\in W_{M^+}$). If, on the other hand, $v,w\in W_{M^-}$
then $v^{-1},w^{-1}, (vw)^{-1}\in W_{M^+}$ and $\mu_{U_P}(vw) =
\mu_{U_P}(v)\cdot \mu_{U_P}(w)$, because \eqref{eq:deltaM} is a group
homomorphism.
Again, Corollary~\ref{cor:mu-prop} implies $q_{v,w} = q_{M,v,w}$.

Notice that $q_{v,w} = q_{M,v,w}$ implies $\hyperplanes_{v}
\cap v\hyperplanes_w = \hyperplanes_{M,v}\cap v\hyperplanes_{M,w}$. Hence, the
second assertion follows from $\ell(vw) = \ell(v) + \ell(w) - 2\cdot \lvert
\hyperplanes_v \cap v\hyperplanes_w\rvert$, \cite[Remark~4.18]{Vigneras.2016},
and a similar formula for $\ell_M(vw)$.
\end{proof} 

\begin{cor}[{\cite[Lemma~4.1]{Abe.2016a}}] 
\label{cor:Abe-2}
Let $w\in W_{M^+}$ (resp. $w\in W_{M^-}$) and $v\in W_M$ with $v\le_Mw$. Then
$v\in W_{M^+}$ (resp. $v\in W_{M^-}$).
\end{cor} 
\begin{proof} 
Replacing $w$ by $w^{-1}$ if necessary, we may assume $w\in W_{M^+}$. Then
Proposition~\ref{prop:mu-prop}.\ref{prop:mu-prop-b} implies $v\in W_{M^+}$.
\end{proof} 

\section{Parabolic induction}\label{sec:parabolic} 
\subsection{Reminder on abstract Hecke algebras}\label{subsec:reminder} 
We recall briefly the basics of Hecke algebras. A thorough introduction with
complete proofs can be found in \cite[Chapter~3, \S1.2]{Andrianov.1995}.

Let $X$ be a topological group and $S\subseteq X$ a submonoid containing a
compact open subgroup $\Gamma$. Then $S$ acts on the right on the free
$\Z$-module on the set of right cosets $\Gamma\backslash S$
\[
\Z[\Gamma\backslash S]\coloneqq \bigoplus_{\Gamma g\in \Gamma\backslash S}
\Z.(\Gamma g).
\]
The $\Z$-module of $\Gamma$-invariants
\begin{equation}\label{eq:Heckering}
H(\Gamma, S) = \Z[\Gamma\backslash S]^\Gamma = \set{t\in \Z[\Gamma\backslash
S]}{t\gamma = t\quad \text{for all $\gamma\in \Gamma$}}
\end{equation}
is free. The element
\[
T_g\coloneqq T_g^S \coloneqq \sum_{\Gamma h} (\Gamma h) \in H(\Gamma,S),
\]
where the sum ranges over all right cosets contained in the double coset $\Gamma
g\Gamma$, depends only on $\Gamma g\Gamma$. Moreover, $(T_g)_{\Gamma g\Gamma \in
\Gamma\backslash S/\Gamma}$ forms a $\Z$-basis of $H(\Gamma,S)$.

The isomorphism
\[
\End_{S}\bigl(\Z[\Gamma\backslash S]\bigr) \stackrel{\cong}{\longrightarrow}
H(\Gamma, S),\quad T\longmapsto T\bigl((\Gamma)\bigr)
\]
endows $H(\Gamma, S)$ with the structure of an associative ring with unit $T_1$.
Concretely, given elements $t = \sum_i a_i\cdot (\Gamma g_i)$ and $t' = \sum_j
a'_j\cdot (\Gamma g'_j)$ in $H(\Gamma, S)$, one has
\begin{equation}\label{eq:explicit-mult}
t\cdot t' = \sum_{i,j} a_ia'_j\cdot (\Gamma g_ig'_j) \in H(\Gamma, S).
\end{equation}

Given a commutative unital ring $R$, the $R$-algebra
\[
H_R(\Gamma, S) \coloneqq R\otimes H(\Gamma, S)
\]
is called the \emph{Hecke algebra} over $R$ associated with $(\Gamma, S)$.

\subsection{Parabolic induction}\label{subsec:parabolic} 
Let $\alg P = \alg M\alg U_{\alg P}$ be a parabolic $\field$-subgroup of $\alg
G$. Recall the pro-$p$ Iwahori subgroup $I_1$ of $G$. Then $I_{1,M} = I_1\cap M$
is a pro-$p$ Iwahori subgroup of $M$ \cite[Lemma~4.1.1]{Haines-Rostami.2009}.
Fix a commutative unital ring $R$. The algebras
\[
\Hecke_R(G)\coloneqq H_R(I_1,G)\qquad \text{and}\qquad \Hecke_R(M)\coloneqq
H_R(I_{1,M},M)
\]
are the usual pro-$p$ Iwahori--Hecke $R$-algebras \cite[(1)]{Vigneras.2016} of
$G$ and $M$, respectively. We write $T_w \coloneqq T_g$ whenever $w\in W(1)$
corresponds to $I_1gI_1$ under \eqref{eq:Bruhat-Decomposition}; a similar
convention applies to $\Hecke_R(M)$. Then $(T_w)_{w\in W(1)}$ and $(T^M_w)_{w\in
W_M(1)}$ are $R$-bases of $\Hecke_R(G)$ and $\Hecke_R(M)$, respectively.

\begin{defn} 
The $R$-algebra $\Hecke_R(P)\coloneqq H_R(I_1\cap P, P)$ is called the
\emph{parabolic pro-$p$ Iwahori--Hecke $R$-algebra}.
\end{defn} 

The main goal of this section will be to construct two $R$-algebra morphisms
\[
\begin{tikzcd}
& \Hecke_R(P) \ar[dl,"\Theta^P_M"'] \ar[dr,"\Xi^P_G"]\\
\Hecke_R(M) & & \Hecke_R(G).
\end{tikzcd}
\]
Pulling back along $\Theta^P_M$ and extending scalars along $\Xi^P_G$ then
defines a functor
\begin{equation}\label{eq:parabolic-induction}
\Mod{\Hecke_R(M)} \longrightarrow \Mod{\Hecke_R(G)},\qquad \module m
\longmapsto \module m\otimes_{\Hecke_R(P)} \Hecke_R(G)
\end{equation}
from the category of right $\Hecke_R(M)$-modules to the category of right
$\Hecke_R(G)$-modules. We then go on to prove that
\eqref{eq:parabolic-induction} is naturally isomorphic to the parabolic
induction functor studied in \cite[(4.2)]{Ollivier-Vigneras.2018} and
\cite{Vigneras.2015}.
\subsubsection{The positive subalgebra} 
Recall the monoid $M^+$ of $M$-positive elements. The algebra
\[
\Hecke_R(M^+)\coloneqq H_R(I_{1,M}, M^+)
\]
is called the \emph{positive subalgebra} of $\Hecke_R(M)$. (The fact that it is
indeed a subalgebra of $\Hecke_R(M)$ is clear from the explicit definition of
the multiplication.)

We collect two well-known and fundamental properties of $\Hecke_R(M^+)$.

\begin{prop}[{\cite[II.5]{Vigneras.1998}}]\label{prop:positive-1} 
Consider the injective $R$-linear map
\[ 
\xi\colon \Hecke_R(M)\longrightarrow \Hecke_R(G),\quad T^M_m \longmapsto T_m.
\] 
The restriction $\xi^+\coloneqq \xi\big|_{\Hecke_R(M^+)}$ respects the
product.
\end{prop} 

\begin{rmk} 
Proposition~\ref{prop:positive-1} relies solely on the Iwahori decomposition of
$I_1$, \ie we could replace $I_1$ by any compact open subgroup $\Gamma$
satisfying $\Gamma = (\Gamma\cap U_{P^\op}).(\Gamma\cap M).(\Gamma\cap U_{P})$.
However, Proposition~\ref{prop:positive-1} fails for groups like $K$ that do
not admit an Iwahori decomposition.
\end{rmk} 

\begin{prop}[{\cite[II.6]{Vigneras.1998}}]\label{prop:positive-2} 
The following assertions are equivalent:
\begin{enumerate}[label=(\alph*)]
\item $\xi^+\colon \Hecke_R(M^+)\to \Hecke_R(G)$ extends to a morphism $\wt
\xi^+\colon \Hecke_R(M)\to \Hecke_R(G)$ of $R$-algebras.
\item There exists a strictly positive element $a\in M$
such that $T_a$ is invertible in $\Hecke_R(G)$.
\end{enumerate}
If one of the assertions holds then $T_a$ is invertible for all strictly
positive elements $a\in M$ and $\wt \xi^+$ is unique.
\end{prop} 
\begin{proof} 
If $a\in M$ is strictly positive then $T^M_a$ is central and invertible in
$\Hecke_R(M)$. Since $(T^M_a)^n\cdot T^M_m = T^M_{a^nm} \in \Hecke_R(M^+)$, for
any $m\in M$ and $n\gg0$, it follows that $\Hecke_R(M)$ is the
localization of $\Hecke_R(M^+)$ at $T^M_a$. The proposition is a formal
consequence of this fact.
\end{proof} 
\subsubsection{The morphism \texorpdfstring{$\Theta^P_M$}{Theta}} 
Let $\Gamma \subseteq P$ be a compact open subgroup satisfying $\Gamma =
\Gamma_M\Gamma_{U_P}$, where $\Gamma_M = \Gamma \cap M$ and $\Gamma_{U_P} =
\Gamma\cap U_P$. For example, $\Gamma$ could be the intersection of $P$ with one
of the groups $K$, $I$, $I_1$.
Recall the notation $g_M \coloneqq \pr_M(g)$, where $\pr_M\colon P\to M$ is the
projection.
\begin{prop}\label{prop:Theta} 
The $R$-linear map
\begin{align*}
\Theta^P_M \coloneqq \Theta^P_{M,R}\colon H_R(\Gamma,P) &\longrightarrow
H_R(\Gamma_M,M),\\
T^P_g &\longmapsto \nu_M(g)\mu_{U_P}(g)\cdot T^M_{g_M}
\end{align*}
is a homomorphism of $R$-algebras.
\end{prop} 
\begin{proof} 
The projection $\pr_M\colon P\to M$ induces an
$R$-linear map $\vartheta\colon R[\Gamma\backslash P] \to R[\Gamma_M\backslash
M]$, given explicitly by $(\Gamma g)\mapsto (\Gamma_Mg_M)$. Since $\pr_M(\Gamma)
\subseteq \Gamma_M$ and $\vartheta$ is right $P$-linear if we let $P$ act by
inflation on $R[\Gamma_M\backslash M]$, it follows that $\vartheta$ maps
$H_R(\Gamma, P) = R[\Gamma\backslash P]^\Gamma$ into $H_R(\Gamma_M,M) =
R[\Gamma_M\backslash M]^{\Gamma_M}$. By restriction we obtain an $R$-linear map
$\theta\colon H_R(\Gamma, P)\to H_R(\Gamma_M,M)$. It is obvious from the
explicit description of multiplication \eqref{eq:explicit-mult} that $\theta$
respects the product. Therefore, it remains to prove $\theta = \Theta^P_M$. Let
$g\in P$. By Proposition~\ref{prop:double-coset} we have
\[
T^P_g = \sum_{i=1}^{\mu_M(g_M)} \sum_{j=1}^{\nu_M(g)}
\sum_{s=1}^{\mu_{U_P}(g)} (\Gamma gu_sh_jm_i)
\]
in $R[\Gamma\backslash P]$, for certain $u_s\in \Gamma_{U_P}$, $h_j\in
(\Gamma_M)_{(g_M)}$, and $m_i\in \Gamma_M$ with $\Gamma_M =
\bigsqcup_{i=1}^{\mu_M(g_M)} (\Gamma_M)_{(g_M)}m_i$.
Since $g_Mh_j \in \Gamma_Mg_M$, by
definition of $h_j$, we obtain
\[
\theta(T^P_g) = \sum_{i,j,s} (\Gamma_Mg_Mh_jm_i) = \sum_{i,j,s}
(\Gamma_Mg_Mm_i) = \nu_M(g)\mu_{U_P}(g)\cdot T^M_{g_M},
\]
where the last equality uses $\Gamma_Mg_M\Gamma_M = \bigsqcup_{i=1}^{\mu_M(g_M)}
\Gamma_Mg_Mm_i$, cf. \cite[Chapter~3, Lemma~1.2]{Andrianov.1995}.
Hence, $\theta$ and $\Theta^P_M$ coincide.
\end{proof} 

For the next two consequences we assume $\Gamma_{U_P} = I_{U_P}$.

\begin{cor}\label{cor:Image-basis} 
The system $(\mu_{U_P}(m) T^M_m)_m$, where
$m$ runs through a system of representatives of $\Gamma_M\backslash M/\Gamma_M$
with $\mu_{U_P}(m)\neq 0$ in $R$, generates $\Image(\Theta^P_{M,R})$ over $R$. If
$R$ is $p$-torsionfree, then it is an $R$-basis.
\end{cor} 
\begin{proof} 
Notice that $\nu_M(m) = 1$ whenever $m\in M$. Moreover, we have
$\mu_{U_P}(g) \ge \mu_{U_P}(g_M)$, for all $g\in P$, by
Proposition~\ref{prop:inequality}. The assertion now follows from
Proposition~\ref{prop:Theta}.
\end{proof} 

\begin{cor} 
The algebra $H_R(\Gamma_M, M^+)$ is contained in
$\Image(\Theta^P_M)$ with equality whenever $qR = 0$.
\end{cor} 
\begin{proof} 
Since $\Gamma_M$ normalizes $\Gamma_{U_P} = I_{U_P}$, we have
$\Gamma_M\subseteq M^+$ and $H_R(\Gamma_M,M^+)$ is defined. Keeping in mind that
$\mu_{U_P}$ takes values in $q^{\Z_{\ge0}}$, the assertion follows from
Corollary~\ref{cor:Image-basis}.
\end{proof} 
\subsubsection{The morphism \texorpdfstring{$\Xi^P_G$}{Xi}} 
Lacking a good reference, we formulate Vign\'eras’ ``fundamental lemma'', a
proof of which appears in \cite[Lemma~13]{Vigneras.2005} (or
\cite[1.2]{Vigneras.2003}) for $\field$-split $\alg G$, and which for general
$\alg G$ is known to the experts.
\begin{lem}[Fundamental Lemma]\label{lem:fundamental} 
Let $v,w\in W(1)$. Then
\[
q_{v,w}\cdot T_v^{-1} T_{vw} - T_w \in \bigoplus_{w'<w} \Z.T_{w'}\qquad \text{in
$\Hecke_\Z(G)$,}
\]
where $<$ denotes the Bruhat order in $W(1)$.
\end{lem} 
\begin{proof} 
Let $o$ be the orientation of $(\apartment,\hyperplanes)$ (cf.
\cite[5.2]{Vigneras.2016}) which places
$\chamber$ in the negative half-space of each hyperplane $H\in \hyperplanes$.
Then $E_o(w) = T_w$ in the notation of \cite[Definition~5.22]{Vigneras.2016}.
Computing inside $\Hecke_{\Z[p^{-1}]}(G)$, we have
\[
q_{v,w}T_v^{-1}T_{vw} = q_{v,w}E_o(v)^{-1}E_o(vw) = E_{o\bullet v}(w) \in
\Hecke_\Z(G),
\]
by \cite[Theorem~5.25]{Vigneras.2016}, and the assertion follows from
\cite[Corollary~5.26]{Vigneras.2016}.
\end{proof} 

\begin{prop}\label{prop:Xi} 
Assume that $R$ is $p$-torsionfree. Then $\xi^+\colon \Hecke_R(M^+)
\to \Hecke_R(G)$ (see Proposition~\ref{prop:positive-1}) extends uniquely to an
injective $R$-algebra morphism
\[
\wt \xi^+\colon \Image(\Theta^P_M) \longrightarrow \Hecke_R(G),
\]
and $\Image(\Theta^P_M)$ is the maximal subalgebra of $\Hecke_R(M)$ with this
property.
\end{prop} 

\begin{defn} 
For arbitrary $R$ we obtain an $R$-algebra morphism
\begin{equation}\label{eq:Xi}
\Xi^P_G \coloneqq \Xi^P_{G,R} \coloneqq \id_R \otimes \bigl(\wt \xi^+ \circ
\Theta^P_{M,\Z}\bigr) \colon \Hecke_R(P) \longrightarrow \Hecke_R(G).
\end{equation}
\end{defn} 

\begin{proof}[Proof of Proposition~\ref{prop:Xi}] 
We view $\Hecke_R(M)$ (resp. $\Hecke_R(G)$) as a subalgebra of
$\Hecke_{R[p^{-1}]}(M)$ (resp. $\Hecke_{R[p^{-1}]}(G)$), which is possible by
our assumption on $R$. Let $a\in M^+$ be a strictly positive element. Then $T_a$
is invertible in $\Hecke_{R[p^{-1}]}(G)$ by \cite[Proposition~4.13
1)]{Vigneras.2016}. By Proposition~\ref{prop:positive-2} the map $\xi^+$ extends
uniquely to an $R[p^{-1}]$-algebra morphism
\[
\wt \xi^+\colon \Hecke_{R[p^{-1}]}(M) \longrightarrow \Hecke_{R[p^{-1}]}(G).
\]
Explicitly, it is given as follows: let $m\in M$, and choose $n\in\Z_{>0}$ such
that $a^nm$ is $M$-positive. Then
\begin{equation}\label{eq:xi-explicit}
\wt\xi^+\bigl(T^M_m\bigr) = T_a^{-n}\cdot \xi^+\bigl(T^M_{a^nm}\bigr) =
T_a^{-n}T_{a^nm}.
\end{equation}
It suffices to prove the following
\begin{claim*} 
The preimage of $\Hecke_R(G)$ under $\wt\xi^+$ coincides with
$\Image(\Theta^P_M)$.
\end{claim*} 

The claim implies the assertion of the proposition: given any extension
$\xi'\colon \mathcal A\to \Hecke_R(G)$ of $\xi^+\colon
\Hecke_R(M^+)\to \Hecke_R(G)$, we have $\xi' = \wt\xi^+\big|_{\mathcal A}$ by
the uniqueness of $\wt\xi^+$, and then $\mathcal A \subseteq
\Image(\Theta^P_M)$ by the claim.

We now prove the claim. By Corollary~\ref{cor:Image-basis} the family
$(\mu_{U_P}(w)T^M_w)_{w\in W_M(1)}$ is an $R$-basis of $\Image(\Theta^P_M)$. As
$a$ is strictly positive, so is the element $\lambda\coloneqq aZ_1 \in
\Lambda(1)$. Given any $w\in W_M(1)$, there exists $n\in\Z_{>0}$ such that
$e^{n\lambda}w \in W_{M^+}(1)$. Now, Corollary~\ref{cor:mu-prop} shows
\[
q_{n\lambda, w} = \frac{\mu_{U_P}(n\lambda)\mu_{U_P}(w)}{\mu_{U_P}
(e^{n\lambda}w)} \cdot q_{M,n\lambda,w} = \mu_{U_P}(w).
\] 
The second equality uses $q_{M,n\lambda}\cdot q_{M,w} = q_{M,e^{n\lambda}w}$,
that means, $q_{M,n\lambda,w} = 1$, which holds because $a$ lies in the center
of $M$. By Lemma~\ref{lem:fundamental} we have
\begin{align}
\wt\xi^+\bigl(\mu_{U_P}(w)T^M_w\bigr) &= \mu_{U_P}(w)\cdot
T^{-n}_{\lambda}T_{e^{n\lambda}w} = q_{n\lambda,w} \cdot T_{n\lambda}^{-1}\cdot
T_{e^{n\lambda}w} \notag\\
&= T_w + \sum_{w'<w}c_{w'}T_{w'} \in \Hecke_R(G), \label{eq:Xi-explicit}
\end{align}
for certain $c_{w'} \in \Z$, viewed as elements of $R$. This shows that
$\wt\xi^+$ is injective and that $\Image(\Theta^P_M)$ is contained in
$(\wt\xi^+)^{-1}(\Hecke_R(G))$.

Conversely, let $T = \sum_{i=1}^k x_i\cdot T^M_{w_i} \in \Hecke_{R[p^{-1}]} (M)$
with $x_i\in R[p^{-1}]\setminus\{0\}$ and $\wt\xi^+(T) \in \Hecke_R(G)$. We
prove $T\in \Image(\Theta^P_M)$ by induction on $k$. The case $k=0$ is trivial.
Assume $k>0$. Rearranging if necessary, we may assume that $w_k\in W_M(1)$ is
maximal among $\{w_1,\dotsc,w_k\}$ with respect to the Bruhat order in $W(1)$.
Let $n\in\Z_{>0}$ with $e^{n\lambda}w_i \in W_{M^+}(1)$ for all $i$. Then
$T^M_{n\lambda}\cdot T = \sum_{i=1}^k x_i\cdot T^M_{e^{n\lambda}w_i}$ lies in
$\Hecke_{R[p^{-1}]}(M^+)$, and hence
\[
\wt\xi^+(T) = \sum_{i=1}^k x_i\cdot T_{n\lambda}^{-1}T_{e^{n\lambda}w_i} =
\sum_{i=1}^k x_iq_{n\lambda,w_i}^{-1}\cdot \Bigl(T_{w_i} + \sum_{w'_i < w_i}
c_{w'_i} T_{w'_i}\Bigr) \in \Hecke_R(G).
\]
Again, we have $q_{n\lambda,w_i} = \mu_{U_P}(w_i)$. Hence, maximality of $w_k$
implies $x_k \in R.\mu_{U_P}(w_k)$, whence $x_kT_{w_k} \in \Image(\Theta^P_M)$.
By the induction hypothesis we have $T - x_kT^M_{w_k} \in \Image(\Theta^P_M)$.
We conclude $T\in \Image(\Theta^P_M)$, finishing the proof.
\end{proof} 

We note the following useful consequence of the proof of
Proposition~\ref{prop:Xi}.
\begin{cor}\label{cor:Xi} 
Let $a\in M$ be strictly positive and $g\in P$ arbitrary. Then
\[
T_{a^n}\cdot \Xi^P_G(T^P_g) = \nu_M(g)\mu_{U_P}(g)\cdot T_{a^ng_M},\qquad
\text{in $\Hecke_R(G)$}
\]
whenever $n\in\Z_{>0}$ is such that $a^ng_M \in M^+$.
\end{cor} 
\begin{proof} 
The assertion follows by extension of scalars from the case $R = \Z$. Thus, it
suffices to prove
\[
T_{a^n}\cdot \wt\xi^+\bigl(\Theta^P_{M,\Z}(T^P_g)\bigr) =
\nu_M(g)\mu_{U_P}(g)\cdot T_{a^ng_M},\qquad \text{in $\Hecke_{\Z}(G)$},
\]
where the computation takes place in $\Hecke_{\Z[p^{-1}]}(G)$.
But this is clear from \eqref{eq:xi-explicit} and the fact that
$\Theta^P_{M,\Z}(T^P_g) = \nu_M(g)\mu_{U_P}(g)\cdot T^M_{g_M}$. 
\end{proof} 
\subsubsection{Equivalence of parabolic inductions} 
Having constructed two morphisms $\Theta^P_M\colon \Hecke_R(P) \to
\Hecke_R(M)$ and $\Xi^P_G\colon \Hecke_R(P)\to \Hecke_R(G)$, we obtain a
functor
\begin{equation}\label{eq:parabolic-induction-2}
\Mod{\Hecke_R(M)} \longrightarrow \Mod{\Hecke_R(G)},\quad \module m \longmapsto
\module m\otimes_{\Hecke_R(P)} \Hecke_R(G)
\end{equation}
from the category of right $\Hecke_R(M)$-modules to the category of right
$\Hecke_R(G)$-modules by viewing $\module m$ via $\Theta^P_M$ as a right
$\Hecke_R(P)$-module and then extending scalars along $\Xi^P_G$.
There is also the parabolic induction, due to
\cite[(4.2)]{Ollivier-Vigneras.2018}, 
\begin{equation}\label{eq:parabolic-induction-3}
\Mod{\Hecke_R(M)} \longrightarrow \Mod{\Hecke_R(G)}, \quad \module m\longmapsto
\module m\otimes_{\Hecke_R(M^+)} \Hecke_R(G),
\end{equation}
given by viewing $\module m$ as a right $\Hecke_R(M^+)$-module and extending
scalars along the $R$-algebra morphism $\xi^+\colon \Hecke_R(M^+)\to \Hecke_R(G)$
(Proposition~\ref{prop:positive-1}). The next theorem is an easy consequence of
the construction of $\Xi^P_G$. 

\begin{thm}\label{thm:parabolic-induction} 
The functors \eqref{eq:parabolic-induction-2} and
\eqref{eq:parabolic-induction-3} are canonically isomorphic.
\end{thm} 
\begin{proof} 
Let $\module n$ be a right $\Hecke_R(G)$-module and let 
\[
\rho\colon \module m \times \Hecke_R(G) \longrightarrow \module n
\]
be an $R$-bilinear map satisfying $\rho(m,TT') = \rho(m,T)T'$ for all $m\in
\module m$ and $T,T'\in \Hecke_R(G)$. The assertion of the theorem is then
tantamount with the equivalence of the following two properties:
\begin{enumerate}[label=(\roman*)]
\item\label{thm:parind-i} $\rho(mT^M, T) = \rho\bigl(m, \xi^+(T^M)T\bigr)$ for
all $m\in \module m$, $T^M\in \Hecke_R(M^+)$, and $T\in \Hecke_R(G)$.

\item\label{thm:parind-ii} $\rho\bigl(m\Theta^P_M(T^P), T\bigr) = \rho\bigl(m,
\Xi^P_G(T^P)T\bigr)$ for all $m\in \module m$, $T^P\in \Hecke_R(P)$, and $T\in
\Hecke_R(G)$.
\end{enumerate}

Given $m\in M^+$, we have $\Theta^P_M(T^P_m) = T^M_m$ and $\Xi^P_G(T^P_m) =
\xi^+(T^M_m)$. Thus, \ref{thm:parind-ii} implies \ref{thm:parind-i}.

Conversely, assume \ref{thm:parind-i} and fix a strictly positive element $a\in
M$. Let $g\in P$ and choose
$n\in\Z_{>0}$ such that $a^ng_M\in M^+$. Then
\begin{align*}
\rho\bigl(m\cdot \Theta^P_M(T^P_g), T\bigr) &= \rho\bigl(m\cdot
\nu_M(g)\mu_{U_P}(g)\cdot T^M_{g_M}, T\bigr)\\
&= \rho\bigl(m\cdot
(T^M_{a^n})^{-1}\cdot \nu_M(g)\mu_{U_P}(g)T^M_{a^ng_M}, T\bigr)\\
&= \rho\bigl(m\cdot (T^M_{a^n})^{-1},
\nu_M(g)\mu_{U_P}(g)T_{a^ng_M}\cdot T\bigr) && \text{(by \ref{thm:parind-i})}\\
&= \rho\bigl(m\cdot (T^M_{a^n})^{-1}, T_{a^n}\cdot \Xi^P_G(T^P_g)\cdot T\bigr)
&& \text{(by Corollary~\ref{cor:Xi})}\\
&= \rho\bigl(m, \Xi^P_G(T^P_g)\cdot T\bigr) & & \text{(by \ref{thm:parind-i})}
\end{align*}
keeping in mind $\xi^+(T^M_m) = T_m$, for all $m\in M^+$. Hence,
\ref{thm:parind-i} implies \ref{thm:parind-ii}.
\end{proof} 

\section{Transitivity of parabolic induction}\label{sec:transitivity} 
We observe that only a proper quotient of the parabolic pro-$p$ Iwahori--Hecke
algebra $\Hecke_R(P)$ affects the parabolic induction functor: both morphisms
$\Theta^P_M$ and $\Xi^P_G$ factor through
\[
R\otimes \Image(\Theta^P_{M,\Z}).
\]
This suggests to study this algebra.

\subsection{Definitions and compatibilities}\label{subsec:definitions} 
\begin{defn} 
We put $\Hecke_R(M,G) \coloneqq R\otimes \Image(\Theta^P_{M,\Z})$. Given $w\in
W_M(1)$, we define
\[
\tau^{M,G}_w \coloneqq 1\otimes \mu_{U_P}(w) T^M_w \in \Hecke_R(M,G).
\]
From Corollary~\ref{cor:Image-basis} it follows that
$\bigl(\tau^{M,G}_w\bigr)_{w\in W_M(1)}$ is an $R$-basis of $\Hecke_R(M,G)$.
Finally, write
\begin{align*}
\theta^{M,G}_M\colon \Hecke_R(M,G) &\longrightarrow \Hecke_R(M),\quad
\text{and}\\
\xi^G_{G,M}\colon \Hecke_R(M,G) & \longrightarrow \Hecke_R(G)
\end{align*}
for the maps induced by $\Theta^P_M$ and $\Xi^P_G$, respectively.
\end{defn} 

\begin{rmk} 
\begin{enumerate}[label=(\alph*)]
\item Although not explicit in the notation, the algebra $\Hecke_R(M,G)$ depends
on $\alg P$. However, in our context $\alg M$ and $\alg P$ determine each other
so that no confusion will arise. 

\item Notice that $\Hecke_R(G,G) = \Hecke_R(G)$.

\item The computation \eqref{eq:Xi-explicit} actually shows
\begin{equation}\label{eq:xi-explicit-2}
\xi^G_{G,M}\bigl(\tau^{M,G}_w\bigr) = T_w + \sum_{w'<w} c_{w'}T_{w'} \in
\Hecke_R(G),
\end{equation}
for all $w\in W_M(1)$. In particular, $\xi^G_{G,M}$ is injective.
\end{enumerate}
\end{rmk} 

\begin{lem}\label{lem:braid} 
Let $v,w\in W_{M}(1)$ with $q_{v,w} = 1$. Then
$\tau^{M,G}_v\cdot \tau^{M,G}_w = \tau^{M,G}_{vw}$.
\end{lem} 
\begin{proof} 
We may assume $R = \Z$. Corollary~\ref{cor:mu-prop} shows
$\mu_{U_P}(v)\cdot\mu_{U_P}(w) = \mu_{U_P}(vw)$ and $q_{M,v,w} = 1$.
Hence, $\tau^{M,G}_v\cdot \tau^{M,G}_w = \mu_{U_P}(v)\mu_{U_P}(w)\cdot
T^M_vT^M_w = \mu_{U_P}(vw)\cdot T^M_{vw} = \tau^{M,G}_{vw}$.
\end{proof} 

\subsubsection{The morphisms \texorpdfstring{$\theta^{L,G}_M$}
{theta(L,G)M}}
\begin{lem}\label{lem:thetaLG} 
Let $\alg M\subseteq \alg L$ be Levi subgroups in $\alg G$. The map
\begin{equation}\label{eq:thetaLG}
\theta^{L,G}_M \colon \Hecke_R(M,G) \longrightarrow \Hecke_R(M,L),\quad
\tau^{M,G}_w \longmapsto \mu_{U_{P_L}}(w)\cdot \tau^{M,L}_w
\end{equation}
is a morphism of $R$-algebras.
Given another Levi subgroup $\alg L'$ containing $\alg L$, the diagram
\[
\begin{tikzcd}
\Hecke_R(M,G) \ar[r,"\theta^{L',G}_M"] \ar[dr,"\theta^{L,G}_M"'] &
\Hecke_R(M,L') \ar[d,"\theta^{L,L'}_M"]\\
& \Hecke_R(M,L)
\end{tikzcd}
\]
commutes, \ie $\theta^{L,G}_M = \theta^{L,L'}_M\circ
\theta^{L',G}_M$.
\end{lem} 
\begin{proof} 
We may assume $R = \Z$. By Proposition~\ref{prop:changing} we compute
\[
\tau^{M,G}_w = \mu_{U_P}(w)T^M_w = \mu_{U_{P_L}}(w)\cdot \mu_{U_P\cap L}(w)T^M_w
= \mu_{U_{P_L}}(w)\cdot \tau^{M,L}_w,
\]
for all $w\in W_M(1)$. Hence, $\theta^{L,G}_M$ is the inclusion map and, in
particular, a morphism of $\Z$-algebras. If $\alg L'$ is another Levi subgroup
containing $\alg L$, then
\begin{align*}
\bigl(\theta^{L,L'}_M\circ \theta^{L',G}_M\bigr)\bigl(\tau^{M,G}_w\bigr) 
&= \theta^{L,L'}_M\bigl(\mu_{U_{P_{L'}} }(w)\tau^{M,L'}_w\bigr)
= \mu_{U_{P_{L'} }}(w)\mu_{U_{P_L}\cap L'}(w)\cdot \tau^{M,L}_w\\
&= \mu_{U_{P_L}}(w)\cdot \tau^{M,L}_w = \theta^{L,G}_M\bigl(\tau^{M,G}_w\bigr),
\end{align*}
for all $w\in W_M(1)$, again by Proposition~\ref{prop:changing}.
\end{proof} 

\begin{prop}\label{prop:localization} 
Let $\alg M\subseteq \alg L$ be Levi subgroups in $\alg G$. Let $\lambda\in
\Lambda(1)$ be a strictly $L$-positive element. Then
\[
\Hecke_R(M,L) \cong \Hecke_R(M,G)\bigl[(\tau^{M,G}_{\lambda})^{-1}\bigr]
\]
and $\theta^{L,G}_M\colon \Hecke_R(M,G)\to \Hecke_R(M,L)$ is the localization
morphism.
\end{prop} 
\begin{proof} 
Notice that $\tau^{M,G}_\lambda$ lies in the center of $\Hecke_R(M,G)$, since
$\lambda$ is lifted by a central element in $L$ (and hence in $M$). Also
$\theta^{L,G}_M(\tau^{M,G}_{n\lambda}) = \tau^{M,L}_{n\lambda}$ is central and
invertible in $\Hecke_R(M,L)$ for each $n\in\Z_{>0}$. Hence, $\theta^{L,G}_M$
induces a well-defined $R$-algebra morphism
\[
\wt\theta^{L,G}_M\colon \Hecke_R(M,G)\bigl[(\tau^{M,G}_{\lambda})^{-1}\bigr]
\longrightarrow \Hecke_R(M,L).
\]
It suffices to construct an $R$-linear inverse. Let $w\in W_{M}(1)$. Choose
$n\in\Z_{>0}$ such that $e^{n\lambda}w\in W_{L^+}(1)$. As $n\lambda$ is lifted
by a central element of $L$, we have $q_{L,n\lambda,w} = 1$. Hence,
Lemma~\ref{lem:braid} shows
\begin{equation}\label{eq:localization-1}
\tau^{M,L}_w = \tau^{M,L}_{-n\lambda}\cdot \tau^{M,L}_{n\lambda}\cdot
\tau^{M,L}_w = \tau^{M,L}_{-n\lambda}\cdot \tau^{M,L}_{e^{n\lambda}w} =
\tau^{M,L}_{-n\lambda}\cdot
\theta^{L,G}_M\bigl(\tau^{M,G}_{e^{n\lambda}w}\bigr).
\end{equation}
Hence, we obtain an $R$-linear map
\[
\gamma\colon \Hecke_R(M,L) \longrightarrow
\Hecke_R(M,G)\bigl[(\tau^{M,G}_{\lambda})^{-1}\bigr],\quad \tau^{M,L}_w
\longmapsto \frac{\tau^{M,G}_{e^{n\lambda}w}}{\tau^{M,G}_{n\lambda}},
\]
which does not depend on the choice of $n$. By \eqref{eq:localization-1} we have
$\wt \theta^{L,G}_M\circ \gamma = \id_{\Hecke_R(M,L)}$. Conversely, let $w\in
W_M(1)$ and $n\in\Z_{>0}$. Take $m\in\Z_{>0}$ with $e^{m\lambda}w \in
W_{L^+}(1)$. As $m\lambda$ is lifted by a central element in $L$, we have
$\mu_{U_P\cap L}(e^{m\lambda}w) = \mu_{U_P\cap L}(w)$. Applying
Proposition~\ref{prop:changing} twice, we compute
\begin{align*}
\mu_{U_{P_L}}(w)\cdot \mu_{U_P}(e^{m\lambda}w) &= \mu_{U_{P_L}}(w)\cdot
\mu_{U_P\cap L}(e^{m\lambda}w)\cdot \mu_{U_{P_L}}(e^{m\lambda}w)\\
&= \mu_{U_{P_L}}(w)\cdot \mu_{U_P\cap L}(w) = \mu_{U_P}(w).
\end{align*}
This shows $\tau^{M,G}_{m\lambda}\cdot \tau^{M,G}_w = \mu_{U_{P_L}}(w)\cdot
\tau^{M,G}_{e^{m\lambda}w}$. Now,
\begin{align*}
(\gamma\circ \wt\theta^{L,G}_M)\left(\frac{\tau^{M,G}_w}{\tau^{M,G}_{n\lambda}}
\right) 
&= \gamma\bigl(\tau^{M,L}_{-n\lambda}\cdot \theta^{L,G}_M(\tau^{M,G}_w)
\bigr) = \mu_{U_{P_L}}(w)\cdot \gamma\bigl(\tau^{M,L}_{-n\lambda}\cdot
\tau^{M,L}_w\bigr)
\\
&= \mu_{U_{P_L}}(w)\cdot \gamma\bigl(\tau^{M,L}_{e^{-n\lambda}w}\bigr) =
\frac{\mu_{U_{P_L}}(w)\cdot \tau^{M,G}_{e^{m\lambda}w}}
{\tau^{M,G}_{(n+m)\lambda}}
= \frac{\tau^{M,G}_{m\lambda}\cdot \tau^{M,G}_w}{\tau^{M,G}_{m\lambda}\cdot
\tau^{M,G}_{n\lambda}} = \frac{\tau^{M,G}_w}{\tau^{M,G}_{n\lambda}}.
\end{align*}
Hence, $\gamma\circ \wt\theta^{L,G}_M =
\id_{\Hecke_R(M,G)[(\tau^{M,G}_{\lambda})^{-1}]}$ finishing the proof.
\end{proof} 

\subsubsection{The morphisms \texorpdfstring{$\xi^G_{L,M}$}{xiG(L,M)}} 
\begin{lem}\label{lem:xiLM} 
Let $\alg M\subseteq \alg L$ be Levi subgroups in $\alg G$. There exists a
unique $R$-algebra morphism
\[
\xi^G_{L,M}\colon \Hecke_R(M,G) \longrightarrow \Hecke_R(L,G)
\]
which is natural in $R$ and satisfies the following property: for all Levi
subgroups $\alg M\subseteq \alg L\subseteq \alg L'$ in $\alg G$, the diagram
\begin{equation}\label{eq:xiLM-commute}
\begin{tikzcd}
\Hecke_R(M,G) \ar[r,"\xi^G_{L,M}"] \ar[d,"\theta^{L',G}_M"'] & \Hecke_R(L,G)
\ar[d,"\theta^{L',G}_L"]\\
\Hecke_R(M,L') \ar[r,"\xi^{L'}_{L,M}"'] & \Hecke_R(L,L')
\end{tikzcd}
\end{equation}
commutes, \ie $\theta^{L',G}_L \circ \xi^G_{L,M} = \xi^{L'}_{L,M} \circ
\theta^{L',G}_M$. Moreover, $\xi^G_{L,M}$ is injective.
\end{lem} 
\begin{proof} 
We first construct a unique morphism $\xi^G_{L,M}$, natural in $R$, making the
diagram
\begin{equation}\label{eq:xiLM-commute-2}
\begin{tikzcd}
\Hecke_R(M,G) \ar[d,"\theta^{L,G}_M"'] \ar[r,"\xi^G_{L,M}"] & \Hecke_R(L,G)
\ar[d,"\theta^{L,G}_L"]\\
\Hecke_R(M,L) \ar[r,"\xi^L_{L,M}"'] & \Hecke_R(L)
\end{tikzcd}
\end{equation}
commutative. Afterwards, we check injectivity and that \eqref{eq:xiLM-commute}
commutes.\medskip

\textit{Step~1:} We prove unique existence provided $R$ is $p$-torsionfree. In
this case, $\theta^{L,G}_M$ and $\theta^{L,G}_L$ are the canonical
inclusions, hence uniqueness is clear. We have to show that $\xi^L_{L,M}$ maps
$\Hecke_R(M,G)$ into $\Hecke_R(L,G)$. Let $w\in W_M(1)$. Recall that by
\eqref{eq:xi-explicit-2} we have
\[
\xi^L_{L,M}\bigl(\tau^{M,L}_w\bigr) = T^L_w + \sum_{w'<_Lw} c_{w'}T^L_{w'} \quad
\text{in $\Hecke_R(L)$},
\]
where $<_L$ denotes the Bruhat order in $W_L(1)$. By
Proposition~\ref{prop:mu-prop}.\ref{prop:mu-prop-b} we have $\mu_{U_{P_L}}(w')
\le \mu_{U_{P_L}}(w)$ for all $w'\in W_L(1)$ with $w'<_Lw$. We deduce that
\begin{align*}
\xi^L_{L,M}\bigl(\tau^{M,G}_w\bigr) &= \mu_{U_{P_L}}(w)\cdot
\xi^L_{L,M}\bigl(\tau^{M,L}_w\bigr) = \mu_{U_{P_L}}(w)T^L_w + \sum_{w'<_Lw}
c_{w'}\mu_{U_{P_L}}(w)T^L_{w'}\\
&= \tau^{L,G}_w + \sum_{w'<_Lw} c'_{w'} \tau^{L,G}_{w'}
\end{align*}
lies in $\Hecke_R(L,G)$. This proves existence of an embedding $\xi^{G}_{L,M}
\colon \Hecke_R(M,G)\to \Hecke_R(L,G)$ making \eqref{eq:xiLM-commute-2}
commutative.\medskip

\textit{Step~2:} We prove unique existence for general $R$. Existence follows
from step~1 by extension of scalars from $\Z$ to $R$. (If $R$ is $p$-torsionfree
this construction coincides with the one in step~1 by the uniqueness
assertion.) We have to prove uniqueness for general $R$. Take a surjection
$f\colon R'\twoheadrightarrow R$ for some $p$-torsionfree ring $R'$ (\eg the
large polynomial ring $\Z[X_r\mid r\in R]$). By naturality of the diagram
\[
\begin{tikzcd}
\Hecke_{R'}(M,G) \ar[d,two heads,"f\otimes \id"'] \ar[r,"\xi^G_{L,M}"] &
\Hecke_{R'}(L,G) \ar[d,two heads, "f\otimes\id"] \\
\Hecke_R(M,G) \ar[r,"\xi^G_{L,M}"'] & \Hecke_R(L,G)
\end{tikzcd}
\]
it follows that $\xi^G_{L,M}$ is uniquely determined by the naturality
requirement.\medskip

\textit{Step~3:} Injectivity of $\xi^G_{L,M}$. By construction we have
\begin{equation}\label{eq:xi-explicit-3}
\xi^G_{L,M}\bigl(\tau^{M,G}_w\bigr) = \tau^{L,G}_w + \sum_{w'<_Lw} c'_{w'}
\tau^{L,G}_{w'},\qquad \text{for all $w\in W_M(1)$,}
\end{equation}
for certain $c'_{w'}\in R$. In particular, $\xi^G_{L,M}$ is injective. \medskip

\textit{Step~4:} Commutativity of~\eqref{eq:xiLM-commute}. By naturality
we may assume $R = \Z$. The outer and lower square in
\[
\begin{tikzcd}
\Hecke_\Z(M, G) \ar[dd, bend right,start anchor=south west, end anchor= north
west, "\theta^{L,G}_M"'] \ar[d,"\theta^{L',G}_M"']
\ar[r,"\xi^G_{L,M}"] & \Hecke_\Z(L,G) \ar[d,"\theta^{L',G}_L"] \ar[dd, bend
left, start anchor=south east, end anchor=north east, "\theta^{L,G}_L"] \\
\Hecke_\Z(M,L') \ar[d,"\theta^{L,L'}_M"'] \ar[r,"\xi^{L'}_{L,M}"] &
\Hecke_\Z(L,L') \ar[d,"\theta^{L,L'}_L"] \\
\Hecke_\Z(M,L) \ar[r,"\xi^L_{L,M}"'] & \Hecke_\Z(L)
\end{tikzcd}
\]
commute by construction. By Lemma~\ref{lem:thetaLG}, and since $\theta^{L,L'}_L$
is injective, the upper square commutes.
\end{proof} 

\begin{lem}\label{lem:xi-composition} 
Let $\alg M\subseteq \alg L\subseteq \alg L'$ be Levi subgroups in $\alg G$.
Then the diagram
\begin{equation}\label{eq:xi-composition}
\begin{tikzcd}
\Hecke_R(M,G) \ar[r,"\xi^G_{L,M}"] \ar[dr,"\xi^G_{L',M}"'] & \Hecke_R(L,G)
\ar[d,"\xi^G_{L',L}"]\\
& \Hecke_R(L',G)
\end{tikzcd}
\end{equation}
commutes, \ie $\xi^G_{L',M} = \xi^G_{L',L}\circ \xi^G_{L,M}$.
\end{lem} 
\begin{proof} 
By naturality it suffices to prove the assertion for $R=\Z$. By naturality, and
since $\Hecke_{\Z}(M', G) \subseteq \Hecke_{\Z[p^{-1}]}(M',G) =
\Hecke_{\Z[p^{-1}]}(M')$ for all Levi subgroups $\alg M'$ in $\alg G$, we may
even assume $R = \Z[p^{-1}]$. Hence, we need to prove commutativity of 
\[
\begin{tikzcd}
\Hecke_{\Z[p^{-1}]}(M) \ar[r,"\xi^L_{L,M}"] \ar[dr,"\xi^{L'}_{L',M}"'] &
\Hecke_{\Z[p^{-1}]} (L) \ar[d,"\xi^{L'}_{L',L}"]\\
& \Hecke_{\Z[p^{-1}]}(L').
\end{tikzcd}
\]
(Notice that $\xi^G_{L,M} = \xi^{L}_{L,M}$ for all Levi subgroups $\alg M
\subseteq \alg L$ in $\alg G$ whenever $p$ is invertible, because then
$\theta^{L,G}_M$ and $\theta^{L,G}_L$ are the identity morphisms in
\eqref{eq:xiLM-commute-2}.)

Let $w\in W_M(1)$ and take a strictly $M$-positive element $\lambda\in
\Lambda(1)$. Let $n\in\Z_{>0}$ with $e^{n\lambda}w\in W_{M^+}(1)$. Then both
$n\lambda$ and $e^{n\lambda}w$ are $M^{+,L}$-, $M^{+,L'}$-, and
$L^{+,L'}$-positive. Hence,
\begin{align*}
\xi^{L}_{L,M}(T^M_w) &= (T^L_{n\lambda})^{-1}\cdot T^L_{e^{n\lambda}w}, &
\xi^{L'}_{L',M}(T^M_w) &= (T^{L'}_{n\lambda})^{-1}\cdot
T^{L'}_{e^{n\lambda}w},\\ 
\xi^{L'}_{L',L}(T^L_{n\lambda}) &= T^{L'}_{n\lambda}, &
\xi^{L'}_{L',L}(T^L_{e^{n\lambda}w}) &= T^{L'}_{e^{n\lambda}w}.
\end{align*}
Therefore, we compute
\[
\bigl(\xi^{L'}_{L',L}\circ \xi^{L}_{L,M}\bigr)\bigl(T^M_{w}\bigr) =
\xi^{L'}_{L',L}\bigl((T^L_{n\lambda})^{-1}\cdot T^L_{e^{n\lambda}w}\bigr) =
(T^{L'}_{n\lambda})^{-1}\cdot T^{L'}_{e^{n\lambda}w} = \xi^{L'}_{L',M}(T^M_w).
\qedhere
\]
\end{proof} 

\subsection{Transitivity of parabolic induction} 
\label{subsec:transitivity}
\begin{prop}\label{prop:tensor} 
Let $\alg M\subseteq \alg L\subseteq \alg L'$ be Levi subgroups in $\alg G$.
The canonical map
\begin{align}\label{eq:tensor-1}
\Hecke_R(M,L') \otimes_{\Hecke_R(M,G)} \Hecke_R(L,G) &\longrightarrow
\Hecke_R(L,L'), \\
x\otimes y &\longmapsto \xi^{L'}_{L,M}(x)\cdot \theta^{L',G}_L(y) \notag
\end{align}
is an isomorphism of $\Hecke_R(M,L')$-$\Hecke_R(L,G)$-bimodules.\footnote{The
maps $\Hecke_R(M,G)\to \Hecke_R(M,L')$ and $\Hecke_R(M,G)\to \Hecke_R(L,G)$ are
the obvious ones, namely $\theta^{L',G}_M$ and $\xi^G_{L,M}$, respectively.
Likewise, the bimodule structure is the obvious one.}
\end{prop} 
\begin{proof} 
The map is well-defined, since \eqref{eq:xiLM-commute} commutes, and preserves
the bimodule structure by definition. Let $\lambda\in \Lambda(1)$ be a strictly
$L'$-positive element. By Proposition~\ref{prop:localization} the map
\eqref{eq:tensor-1} identifies with the canonical map
\[
\Hecke_R(M,G)\bigl[(\tau^{M,G}_\lambda)^{-1}\bigr] \otimes_{\Hecke_R(M,G)}
\Hecke_R(L,G) \longrightarrow
\Hecke_R(L,G)\bigl[(\tau^{L,G}_\lambda)^{-1}\bigr],
\]
which is clearly an $R$-linear isomorphism.
\end{proof} 

\begin{thm}\label{thm:tensor} 
Let $\alg L \subseteq \alg L'$ and $\alg M\subseteq \alg M'\subseteq \alg M''$
be Levi subgroups in $\alg G$ with $\alg M\subseteq \alg L$ and $\alg
M'\subseteq \alg L'$. Then the canonical map
\begin{align*}
\Hecke_R(M,L) \dotimes_{\Hecke_R(M,G)} \Hecke_R(M'',G) &\longrightarrow
\Hecke_R(M,L) \dotimes_{\Hecke_R(M,L')} \Hecke_R(M',L') \dotimes_{\Hecke_R(M',G)}
\Hecke_R(M'',G),\\
x\dotimes y &\longmapsto x\dotimes 1\dotimes y
\end{align*}
is an isomorphism of $\Hecke_R(M,L)$-$\Hecke_R(M'',G)$-bimodules.
\end{thm} 
\begin{proof} 
There are natural isomorphisms
\begin{align*}
\Hecke_R(M,L)&\dotimes_{\Hecke_R(M,G)} \Hecke_R(M'',G)\\ 
&\cong \Hecke_R(M,L) \dotimes_{\Hecke_R(M,L')} \Hecke_R(M,L')
\dotimes_{\Hecke_R(M,G)} \Hecke_R(M',G) \dotimes_{\Hecke_R(M',G)}
\Hecke_R(M'',G)\\
&\cong \Hecke_R(M,L)\dotimes_{\Hecke_R(M,L')} \Hecke_R(M',L')
\dotimes_{\Hecke_R(M',G)} \Hecke_R(M'',G),
\end{align*}
the second isomorphism being given by Proposition~\ref{prop:tensor}. The
composite sends $x\otimes y \mapsto x\otimes 1\otimes y$.
\end{proof} 

As an application we give another proof of the transitivity of parabolic
induction, which is originally due to Vign\'eras
\cite[Proposition~4.3]{Vigneras.2015}.

\begin{cor}\label{cor:transitivity} 
Let $\alg M\subseteq \alg L$ be Levi subgroups in $\alg G$. Let $\module m$ be a
right $\Hecke_R(M)$-module. Then there is a natural isomorphism of right
$\Hecke_R(G)$-modules
\[
\module m\otimes_{\Hecke_R(M^{+,L})} \Hecke_R(L) \otimes_{\Hecke_R(L^+)}
\Hecke_R(G) \cong \module m\otimes_{\Hecke_R(M^+)} \Hecke_R(G).
\]
\end{cor} 
\begin{proof} 
Since there is a natural right $\Hecke_R(G)$-linear isomorphism 
\[
\module m\otimes_{\Hecke_R(M^+)} \Hecke_R(G) \cong \module m
\otimes_{\Hecke_R(M)} \Hecke_R(M) \otimes_{\Hecke_R(M^+)} \Hecke_R(G)
\]
(and similarly with $(M^+,G)$ replaced by $(M^{+,L},L)$), we are reduced to
proving the assertion for $\module m = \Hecke_R(M)$. Now, by
Theorems~\ref{thm:parabolic-induction} and~\ref{thm:tensor} there are
$\Hecke_R(M)$-$\Hecke_R(G)$-bimodule isomorphisms
\begin{align*}
\Hecke_R(M)\otimes_{\Hecke_R(M^+)} \Hecke_R(G) &\cong \Hecke_R(M)
\otimes_{\Hecke_R(M,G)} \Hecke_R(G)\\
&\cong \Hecke_R(M) \otimes_{\Hecke_R(M,L)} \Hecke_R(L)
\otimes_{\Hecke_R(L,G)} \Hecke_R(G)\\
&\cong \Hecke_R(M)\otimes_{\Hecke_R(M^{+,L})} \Hecke_R(L)
\otimes_{\Hecke_R(L^+)} \Hecke_R(G).\qedhere
\end{align*}
\end{proof} 

\subsection{Alcove walk bases and a filtration}\label{subsec:filtration} 
We finish by describing a natural $\Z_{\ge0}$-filtration on the $R$-algebra
$\Hecke_R(M,G)$ coming from $\mu_{U_P}\colon W_M(1)\to \Z_{\ge0}$. To do this we
need to describe alcove walk bases for $\Hecke_R(M,G)$.

\begin{defn} 
Let $o$ be an orientation of $(\apartment_M,\hyperplanes_M)$
\cite[5.2]{Vigneras.2016}. Let $(E_o(w))_{w\in W_M(1)}$ be the associated alcove
walk basis in $\Hecke_\Z(M)$ \cite[Definition~5.22]{Vigneras.2016}. We define
\begin{equation}\label{eq:alcove-walk}
E_o^{M,G}(w) \coloneqq 1\otimes \mu_{U_P}(w)\cdot E_o(w) \in
\Hecke_R(M,G),\qquad \text{for all $w\in W_M(1)$.}
\end{equation}
\end{defn} 

\begin{rmk} 
The element $E_o^{M,G}(w)$ is indeed well-defined: since $E_o(w) = T^M_w +
\sum_{w'<_Mw} c_{w'}T^M_{w'}$, for certain $c_{w'}\in \Z$,
\cite[Corollary~5.26]{Vigneras.2016} and by
Proposition~\ref{prop:mu-prop}.\ref{prop:mu-prop-b}, we even have
\begin{equation}\label{eq:alcove-walk-2}
E^{M,G}_o(w) = \tau^{M,G}_w + \sum_{w'<_Mw} c'_{w'}\cdot
\tau^{M,G}_{w'} \in \Hecke_R(M,G),
\end{equation}
where $c'_{w'}$ is the image of $\frac{\mu_{U_P}(w)}{\mu_{U_P}(w')}\cdot c_{w'}$
in $R$. Hence, $(E^{M,G}_o(w))_{w\in W_M(1)}$ is an $R$-basis of
$\Hecke_R(M,G)$.
\end{rmk} 

\begin{lem}\label{lem:alcove-walk-product} 
Let $o$ be an orientation of $(\apartment_M, \hyperplanes_M)$. Let $v,w\in
W_M(1)$. Then
\[
E^{M,G}_o(v)\cdot E^{M,G}_{o\bullet v}(w) = q_{v,w}\cdot E^{M,G}_o(vw).
\]
\end{lem} 
\begin{proof} 
We may assume $R = \Z$. Then $E_o(v)\cdot E_{o\bullet v}(w) = q_{M,v,w}\cdot
E_o(vw)$ by \cite[Theorem~5.25]{Vigneras.2016}. Corollary~\ref{cor:mu-prop}
shows
\[
E^{M,G}_o(v)\cdot E^{M,G}_{o\bullet v}(w) =
\frac{\mu_{U_P}(v)\mu_{U_P}(w)}{\mu_{U_P}(vw)}\cdot q_{M,v,w} E^{M,G}_o(vw) =
q_{v,w}\cdot E^{M,G}_o(vw). \qedhere
\]
\end{proof} 

\begin{defn} 
A \emph{$\Z_{\ge0}$-filtration} of an $R$-algebra $A$ is a family $(\filtration_iA)_{i\in\Z_{\ge0}}$ of $R$-submodules satisfying
\begin{enumerate}[label=$-$]
\item $\filtration_iA \subseteq \filtration_{i+1}A$ for all $i\ge0$;
\item $\filtration_iA\cdot \filtration_jA \subseteq \filtration_{i+j}A$ for
all $i,j\ge0$; 
\item $1\in \filtration_0A$; 
\item $A = \cup_{i\ge0}\filtration_iA$.
\end{enumerate}
\end{defn} 

\begin{prop}\label{prop:H(M,G)-filtration} 
The free $R$-submodules $\filtration^{M,G}_n$ of $\Hecke_R(M,G)$ generated by
$\{\tau^{M,G}_w\}_{\substack{w\in W_M(1),\\ \mu_{U_P}(w)\le q^n}}$ define a
$\Z_{\ge0}$-filtration on $\Hecke_R(M,G)$. Moreover,
$\filtration^{M,G}_0 \cong \Hecke_R(M^+)$ via $\theta^{M,G}_M$.
\end{prop} 
\begin{proof} 
The only thing that is not immediately clear is $\filtration^{M,G}_i\cdot
\filtration^{M,G}_j \subseteq \filtration^{M,G}_{i+j}$, for $i,j\ge0$. Given any
orientation $o$ of $(\apartment_M,\hyperplanes_M)$, the set
$\bigl\{E^{M,G}_o(w)\bigr\}_{\substack{w\in W_M(1)\\ \mu_{U_P}(w)\le q^n}}$ is
an $R$-basis of $\filtration^{M,G}_n$ by \eqref{eq:alcove-walk-2} and
Proposition~\ref{prop:mu-prop}.\ref{prop:mu-prop-b}. Hence, the
claim follows from Lemma~\ref{lem:alcove-walk-product} and
Corollary~\ref{cor:mu-prop}.
\end{proof} 

\bibliographystyle{alphaurl}
\bibliography{../references}{}
\end{document}